\documentclass[12pt, reqno]{amsproc}
\usepackage{amsmath,amssymb,mathrsfs,hyperref,color}

\DeclareMathAlphabet{\mathpzc}{OT1}{pzc}{m}{it}
\usepackage[nameinlink]{cleveref}
\usepackage{mathtools}
\usepackage[x11names]{xcolor}
\usepackage{paralist}

\usepackage{subfigure}
\usepackage{float}
\usepackage{palatino}
\usepackage{tikz}
\usepackage{cases}
\usetikzlibrary{intersections}
\usepackage{xcolor}
\usepackage[latin1]{inputenc}
\usepackage{bbm}
\usepackage{tabularx}
\usepackage{booktabs}
\usepackage[labelfont=bf,format=plain,justification=raggedright,singlelinecheck=false]{caption}
\usetikzlibrary{shadings}
\usetikzlibrary[intersections]
\usetikzlibrary{arrows,shapes}

\makeatletter
\def\setliststart#1{\setcounter{\@listctr}{#1}%
  \addtocounter{\@listctr}{-1}}
\makeatother

\makeatletter
\@addtoreset{figure}{section}
\makeatother

\usepackage[mathcal]{euscript}
\usepackage{calc}
\newtheorem{theorem}{{\bf Theorem}}[section]
\newtheorem{lemma}[theorem]{Lemma}
\newtheorem{proposition}[theorem]{Proposition}

\newtheorem{remarks}[theorem]{Remark}
\newtheorem{example}[theorem] {Example}
\newtheorem{definition}[theorem]{Definition}
\newtheorem{assumption}[theorem]{Assumption}

\numberwithin{equation}{section}

\newcommand{\R}{\mathbb{R}}

\newcommand{\E}{\mathbb{E}}
\newcommand{\N}{\mathbb{N}}

\newcommand{\GG}{\mathsf{G}}

\newcommand{\SR}{\mathbf{CC}}
\newcommand{\X}{\mathcal{X}}
\newcommand{\D}{\mathcal{D}}

\newcommand{\g}{\mathfrak{g}}

\renewcommand{\div}{\mbox{\rm div}}

\renewcommand{\P}{\mathbb{P}}

\def\acc{\`}

\DeclareMathOperator*{\eps}{\varepsilon}

\makeatletter
\def\moverlay{\mathpalette\mov@rlay}
\def\mov@rlay#1#2{\leavevmode\vtop{%
   \baselineskip\z@skip \lineskiplimit-\maxdimen
   \ialign{\hfil$\m@th#1##$\hfil\cr#2\crcr}}}
\newcommand{\charfusion}[3][\mathord]{
    #1{\ifx#1\mathop\vphantom{#2}\fi
        \mathpalette\mov@rlay{#2\cr#3}
      }
    \ifx#1\mathop\expandafter\displaylimits\fi}
\makeatother

\definecolor{ao}{rgb}{0.0, 0.5, 0.0}
\definecolor{brown}{rgb}{0.59, 0.29, 0.0}

\title[Fokker-Planck equations and probabilistic counterparts]{Fokker-Planck equations on homogeneous Lie groups\\ and probabilistic counterparts}
\author{Lucia Caramellino\and Cristian Mendico}
\address{Dipartimento di matematica, Universit\acc a degli studi di Roma Tor Vergata -- Via della Ricerca Scientifica 1, 00133 Roma}
\email{caramell@mat.uniroma2.it} 
\address{Institut de Math\'ematique de Bourgogne, UMR 5584 CNRS, Universit\'e Bourgogne, 21000 Dijon, France}
\email{cristian.mendico@u-bourgogne.fr}
\date{\today}
\thanks{{\it Acknowledgement:} C.M. is partially supported by Istituto Nazionale di Alta Matematica, INdAM-GNAMPA project 2023 and the King Abdullah University of Science and Technology (KAUST) project CRG2021-4674 ``Mean-Field Games: models, theory and computational aspects". Both authors acknowledge  the MIUR Excellence Department Project MatMod@TOV awarded to the Department of Mathematics, University of Rome Tor Vergata, CUP E83C23000330006. \\{\bf Disclosure statement.} The authors report there are no competing interests to declare.}

\begin{document}
%\usetagform{blue}
\begin{abstract}
We address the well-posedness of subelliptic Fokker-Planck equations arising from stochastic control problems, as well as the properties of the associated diffusion processes.  Here, the main difficulty arises from the possible polynomial growth of the coefficients, which is related to the growth of the family of vector fields generating the first layer of the associated Lie algebra. We prove the existence and uniqueness of the energy solution, and its representation as the transition density of the underlying subelliptic diffusion process. Moreover, we show its H\"older continuity  in time w.r.t. the Fortet-Mourier distance, where the H\"older seminorm depends on the degree of homogeneity of the vector fields. Finally, we provide a probabilistic proof of the Feyman-Kac formula,  as a consequence of the uniform boundedness in finite time intervals of all moments.
\end{abstract}
\maketitle

%\tableofcontents

\def\D{D}

\section{Introduction}

%The goal of this paper is twofold: we are interested in the analysis of Fokker-Planck equations posed on homogeneous Lie groups and in the probabilistic consequences on the diffusion process defined on such anisotropic structure. The interest in PDEs and diffusion processes defined on Carnot-Caratheodory spaces (or, more generally, on sub-Riemannian manifolds) has grown rapidly in the last years. 

Real models, possibly involving stochasticity, can be expressed in terms of nonholonomic structures. As an example, albeit limited, we recall neural models \cite{Neuro} (and references therein),  deep neural networks \cite{Teichman},  variational formulations for likelihood estimation  in statistics \cite{NEW, NEW1, Stat}, mean field models  \cite{Mannucci2,Mendico_MathAnn,Mannucci1} and subelliptic regularizations for ergodic control problems as in \cite{Mendico_JDE}. 

Homogeneous Lie groups are special nonholonomic
structures and, following the analysis developed in \cite{Mendico_MathAnn}, 
we study here Fokker-Planck equations and the probabilistic consequences %on the underlying diffusion processes defined 
on such anisotropic structures. In particular, we can deal with vector fields having possible polynomial growth and, up to the authors knowledge, no results are available for stochastic control problems when the underlying diffusion process is driven by coefficients which are not  sublinear and driven by a drift which is merely locally bounded in time.

\subsection*{Problem and state of the art.}
Here, the anisotropic structure of the Euclidean space $\mathbb{R}^d$ is the one endowed by the  Carnot-Carath\'eodory distance induced by the vector fields $X_1, \dots,$ $ X_m$ with $m < d$ (see \eqref{dSR}).
%. This actually  creates the main issues in the analysis of the Fokker-Planck equation and the underlying diffusion process. 
% Indeed, the ellipticity directions span subspaces of dimension strictly less than the dimension of the state space at any point. To overcome this issue, 
We suppose that ${X_1, \dots, X_m}$ satisfy the H\"ormander condition \eqref{Hor} and are the left-invariant vector fields associated with a basis of  the first layer of a homogeneous Lie group $\GG$. We also  assume that ${X_1, \dots, X_m}$ are $1$-homogeneous w.r.t.  the natural dilation function associated with the group structure (we can actually deal with general positive homogeneous coefficients, and we will explain how things will change). We refer to \Cref{Lie} for a detailed description of the properties of the vector fields and of the group structure. 
In particular,  \Cref{example} provides a non trivial example of vector fields and of the associated dilation function (showing that the coordinate forms  are not necessarily linear).

We denote by $\Delta_{\GG}$ and $\div_\GG$ the Laplace operator and the divergence operator on $\GG$ respectively: 
$$
\Delta_{\GG} u = \sum_{i=1}^{m} X^2_i u 
\quad\mbox{and}\quad
\div_{\GG} v = X_1 v_1 + \dots + X_m v_m,
$$
where $u:\R^d\to \R$ and $v:\R^d\to \R$. Then, for $s \geq 0 $ and $x \in \R^d$, we study, both from an  analytical and probabilistic viewpoint, the subelliptic Fokker-Planck equation 
\begin{equation}\label{FP_intro}
\begin{cases}
\partial_t \rho(t, y) - \Delta_{\GG} \rho(t, y) - \div_{\GG}(\beta(t, y) \rho(t, y)) = 0, & (t, y) \in (s, \infty) \times \R^d
\\
\displaystyle
\rho(t, y)dy\rightharpoonup
\delta_{\{x\}}(dy)
\mbox{ as }{t \downarrow s},
\end{cases}
\end{equation}
% (see \eqref{DeltaG} and \eqref{divG}). 
in which $\beta: [0,\infty) \times \R^d \to \R^m$ is a smooth vector field, $\delta_{\{x\}}$ denotes the Dirac mass at $x$ and $\rightharpoonup$ denotes weak convergence. Hence, we are considering the fundamental solution, allowing one to represent the solution to the Fokker-Plank equation for \textit{any} initial datum.

Fokker-Planck equations defined on ani\-so\-tropic structures have been widely studied. Equation  \eqref{FP_intro} differs from the one typically studied in the literature for the presence of a time dependent drift term $\beta$ with weaker regularity. For example, in \cite{Sub2, Sub3, Sub4, Sub5} the authors address some regularity properties, such as Harnack's inequality and Shauder type estimates, of the fundamental solution to the subelliptic heat equation in the case of $\beta\equiv 0$ or $\beta$ independent of time. For the case of time independent or constant drift $\beta$ we also refer to\cite{Lanconelli, Anceschi, Pascucci} where the authors first indicate the applications to finance and economics as motivation of their work. A general and self-contained introduction to the subject can be also found for instance in the  monographs  \cite{Sub1, BLU, Bramanti_2010}.

Moreover, it is worth mentioning that our problem can be tackled with a different perspective, which makes use of Dirichlet forms. As an example, we mention the study of the transition probability function associated with the infinitesimal generator of a time dependent Dirichlet form in \cite{Kim}, where, however, the uniform ellipticity is assumed and thus it does not apply to our case. As a further example, in the comprehensive papers \cite{Saloff} and \cite{Saloff1} the authors deal with 
Harnack's inequality, regularity results, existence and size estimates of the Green's function for more general structures, such as symmetric and non-symmetric Dirichlet forms defined on metric measure spaces.
%, can be found, for instance, in \cite{Saloff}, \cite{Saloff1} and references therein. 
However, our problem is not covered by these references. Indeed, the coefficients of the time-dependent Dirichlet form are assumed to be globally bounded in time, and this is not our case of interest. In fact, due to the examples from stochastic control problems for instance in \cite{Mendico_MathAnn}, we  assume here that $\beta$ is merely locally bounded in time. 

\subsection*{Results.}
As a first step, we consider the PDE in \eqref{FP_intro} when the initial datum is a probability measure having bounded density w.r.t. the Lebesgue measure (see \eqref{FP_regular}). In such a case, by using a purely PDE approach, in \Cref{thm:ex_FP1} we prove  that there exists a unique energy solution. Moreover, such solution remains bounded for all time and it is $\frac{1}{2}$-H\"older continuous in time w.r.t. the Fortet-Mourier distance tailored to the Lie group structure (for the definition, see \eqref{dFM}). 
Afterwards, by developing an approximation argument, in \Cref{delta} we state the same results for the Fokker-Plank equation \eqref{FP_intro}.

Let us stress that the notion of energy solution (see \Cref{def2}) is tailored to our anisotropic structure of the state space. Indeed, the classical integration by parts used to get weak solutions 
%of the Fokker-Planck equation \eqref{FP_intro}, %multiplied by a test function 
would give the extra-terms $\div_{\R^d} (X_i)$ (Euclidean divergence operator, see \eqref{div_euclidea}). In order to deal with those terms one should look at the coefficients, in a coordinate chart, of the vector fields and one should work under more restrictive assumptions on them.
% to get the well-posedness of the equation. 
In our case, as $X_1, \dots, X_m$ are left invariant and homogeneous w.r.t. the natural dilation function associated with the sub-Riemannian structure, we can overcome such an issue. We also stress that, as for the classical Euclidean case, we prove the uniqueness result assuming the following minimal request:  for every $T>s$,
\begin{equation}\label{negative_intro}
\sup_{t \in [s, T]} \| [\div_{\GG} \beta(t, \cdot)]_{-}\|_{(L^{\infty}(\R^d))^m}<\infty.
\end{equation}
This is a classical condition which ensures the well-posedness and the validity of stability estimates in Lebesgue spaces for inviscid transport and continuity equations in the Euclidean space, along with a suitable growth of the coefficients.  
Moreover, we point out that our approach is purely intrinsic and does not involve the Euclidean representation of the Lie group structure. Such a metric approach allows us to overcome the issue of polynomial growth of the diffusion coefficients and the lack of compactness of the state space.

Once the analytical results are achieved, we move to their probabilistic counterparts. As the limiting initial datum in \eqref{FP_intro} is a Dirac mass, we are actually dealing with the so-called fundamental solution of the parabolic PDE -- classically considered -- represented as the formal dual of \eqref{FP_intro}. In fact,  \Cref{representation} states that the solution to  \eqref{FP_intro} is actually the transition density of the solution $\xi$ to the stochastic differential equation (SDE)
\begin{equation}\label{diffusion_intro}
\begin{cases}
d\xi_t = X_0(t,\xi_t)\;dt + \sqrt{2}\; \sum_{i=1}^{m} X_i(\xi_t) \circ dB^i_t
\\
\xi_s=x
\end{cases}
\end{equation}
where $B=(B^1, \dots, B^m)$ is a standard Brownian motion on $\R^m$,  $\circ$ denotes the Stratonovi\-ch integral and 
\begin{equation}\label{diffusion_intro_beta}
X_0(t,\xi) = -\sum_{i=1}^{m} \beta_i(t,\xi) X_i(\xi),
\end{equation}
$X_i(x)$ denoting the coordinate form of $X_i$. Despite the polynomial growth of the coefficient in the SDE \eqref{diffusion_intro}, under our requests the stochastic process $\xi$ does not explode (as a natural counterpart of the mass conservation of the energy solution to \eqref{FP_intro}). Furthermore, its moments are all bounded uniformly in finite time intervals (\Cref{Ito}), so that the solution to \eqref{FP_intro} has all moments, a property which doesn't seem to be achievable  with pure analytical arguments. We point out that, from the probability point of view, since $\beta$ is not constant in the space variable the existence of the transition density for \eqref{diffusion_intro} is a significant result, as it does not follow from classical  arguments in stochastic calculus (see \Cref{rem-diff} for details). Finally, as the subelliptic process $\xi_t$ has finite moments uniformly in finite time intervals, in \Cref{Feynman-Kac} we provide the Feynman-Kac formula (see e.g.  \cite{Fundamental, Schauder, Uniqueness}, where purely analytical arguments are used).

\subsection*{Applications.}
Our analysis and results  have a natural application to stochastic control problems with nonholonomic constraints. We mention the work developed in \cite{Federica, Ermal, Mendico_MathAnn, Mendico_JDE}, where the authors study the Hamilton-Jacobi equation (classically solved by the value function associated with the underlying variational problem) defined on homogeneous Lie groups and the Fokker-Planck equation driven by the optimal feedback as in the mean field games theory. The approach used therein and, in general, for the study of other multi-agent models (e.g., \cite{Multi1} and \cite{Filippo}) rely solely on the analytical study of the equations involved. Indeed, 
%	going back to the assumptions on $\beta$ in \eqref{diffusion_intro} and \eqref{diffusion_intro_beta} 
	as an example we can consider the minimization problem
\begin{equation}\label{value_function}
u(x, t) = \inf_{\alpha} \mathbb{E}\left[\int_{t}^{T}\left(|\alpha(s, \xi_s)|^{\gamma} + F(\xi_s)\right)\;ds + G(\xi_T) \right]
\end{equation}
where $\alpha$ is a square integrable control function and $\xi$ satisfies the stochastic equation \eqref{diffusion_intro} with 
	$$X_0(t,\xi_t) = -\sum_{i=1}^{m} \alpha_i(t,\xi_t) X_i(\xi_t).$$

%\begin{equation*}
%d\xi_t = X_0(t, \xi_t)\;dt + \sqrt{2}\; \sum_{i=1}^{m} X_i(\xi_t) \circ dB^i_t \mbox{ with } X_0(t,\xi_t) = -\sum_{i=1}^{m} \alpha_i(t,\xi_t) X_i(\xi_t),
%\end{equation*}
%$\xi_t= x$ and $\alpha$ is a square integrable control function. 
In this case, it is known that the optimal strategy $\overline\alpha$ is given by
\begin{equation*}
\overline\alpha(s, \xi_s) = \gamma^* |\nabla_{\GG}u(\xi_s, s)|^{\gamma^*-2}\nabla_{\GG}u(\xi_s, s) \quad \forall\; s \in (t, T) 
\end{equation*}
where $\gamma^*$ is the conjugate index of $\gamma$. 
Thus, as shown for instance in \cite{Mendico_MathAnn}, such optimal control map satisfies the required assumptions on $\beta=\overline\alpha$ except for \eqref{negative_intro} related to the horizontal semiconcavity of the value function \eqref{value_function} which still remains an open question to be investigated. Our work aims to introduce a novel approach to such problems, and more generally, to stochastic control problems with nonholonomic constraints by using the dynamics of underlying controlled subelliptic stochastic process. This motivation also underlies the specific structure of the Fokker-Planck equation considered here. 

From the statistical viewpoint, referring to  \cite{NEW, NEW1, Stat} and references therein, we stress that the geometric structure defined by the coefficients of the diffusion process is of Riemannian type. Such assumption now could be relaxed assuming to work on an homogeneous Lie group and, thus, admitting degenerate structures. We plan to address such a problem in a future work.

\subsection*{Organization of the paper} \Cref{Lie} is devoted to the introduction of the sub-Riemannian setting and to the main assumptions in our framework. In \Cref{PDE} we address the problem of the existence, uniqueness and regularity of classical and energy solutions to the subelliptic Fokker-Planck equation. We conclude by analyzing the probabilistic counterpart in \Cref{Proba}, namely we deduce the non explosion of the underlying diffusion process, the existence of all moments, the link with the transition probability density and the Feynman-Kac formula.

\section{Homogeneous Lie groups, horizontal vector fields  and standing assumptions}\label{Lie}

For a detailed analysis of homogeneous vector fields and their properties we refer to \cite{Ancona_96}, and for more on sub-Riemannian geometry and Lie groups we refer to \cite{Barilari}, \cite{Jean}, \cite{Montgomery} and \cite[Chapter 3]{FisherRuzhansky}. In this section we introduce our setting, we state our main assumptions and we present the main results needed throughout the manuscript. 

\subsection{Homogeneous Lie groups}\label{sect:Lie}
Let $(\GG, *)$ be a homogeneous Lie group, which is a simply connected Lie group whose Lie algebra $\g$ is equipped with a family of dilations $\{D_t\}_{t>0}$ (see e.g. \cite[Section 3.1.2]{FisherRuzhansky}). 
We recall that $\{D_t\}_{t>0}$ are automorphisms of $\g$, i.e.,
\begin{equation}\label{dilations}
D_t[g_1, g_2]=[\D_tg_1, \D_tg_2], \quad g_1, g_2 \in \g,\; t > 0,
\end{equation}
The dilations \eqref{dilations} induce a direct sum decomposition on $\g$: 
\begin{equation}\label{direct}
\g = V_1 \oplus \dots \oplus V_\kappa
\end{equation}
and the subspaces $V_1,\ldots,V_\kappa$ satisfy the following property: there exist $\lambda_\kappa>\lambda_{\kappa-1}>\cdots>\lambda_1>0$ such that
\begin{equation}\label{DIL}
D_tg=t^{\lambda_{\ell}}g,\quad \mbox{for every $g\in V_\ell$,  $\ell=1,\ldots,\kappa$.}
\end{equation}	
In particular, $V_1$ is called the horizontal layer. We can identify $(\GG, *)$ with $(\R^d, *)$ via the so-called exponential map 
\begin{equation*}
exp: \g \to \GG,
\end{equation*}
see, for instance, \cite[Section 3.1.3]{FisherRuzhansky}. Indeed, given a basis $e_1,\ldots,e_d\in \g$, any $x \in \GG$ can be written in a unique way as 
\begin{equation*}
x= exp(x_1e_1+\dots + x_de_d)
\end{equation*}
which yields the identification of $x$ with $(x_1, \dots, x_d)$ and that of $(\GG, *)$ with $(\R^d, * )$, the group law $* $ being given by the Baker-Campbell-Hausdorff formula:
$$
x*  y=x+y+\frac 12 [x,y]+\frac 1{12}[x,[x,y]]+\frac 1{12}[y,[y,x]]+\cdots.
$$
Notice that, since \eqref{DIL} implies that $\g$ is nilpotent, the Baker-Campbell-Hausdorff formula actually involves a finite number of iterated brackets.

Moreover, we can transport the dilations to the group using the exponential mapping, that is, 
\begin{equation*}
\big\{exp \circ \D_t \circ exp^{-1}\big\}_{t > 0} 
\end{equation*}
are automorphism of the group $\GG$ (\cite[Section 3.1.2]{FisherRuzhansky}). As a consequence, we get the following:
if $\pi_i$, $i =1, \dots, \kappa$, denote the projection functions onto the layer $V_i$, $i =1, \dots, \kappa$, 
then, by \eqref{DIL}, we can write
\begin{equation*}
\D_{t}x = (t^{\lambda_1} \pi_1(x), \dots, t^{\lambda_\kappa}\pi_\kappa(x)),\quad x\in\R^d,\ t>0,
\end{equation*}
where $\lambda_\kappa>\cdots>\lambda_1>0$ are the weights of the dilation function $\D_t$.

We finally introduce the homogeneous norm $\|\cdot \|_{\GG} $ and the homogeneous dimension $Q$ of the group $\GG$, which are defined as
\begin{equation}\label{hom-norm-dim}
\|x\|_{\GG} = \left(\sum_{j=1}^{k}|\pi_j(x)|^{\frac{2k!}{j}} \right)^{\frac{1}{2k!}} \quad \text{and} \quad Q =\sum_{j=1}^{k} j\  \text{dim}\ V_j
\end{equation}
respectively.

\subsection{Horizontal vector fields}\label{sect:vf}
Based on the setup in \Cref{sect:Lie}, from now on we will  work on $\R^d$ and we will consider the set $\{X_1, \dots, X_m\}$ of the left-invariant vector fields associated with a basis $e_1,\ldots, e_m$  of $V_1$, that is,
\begin{equation}\label{left-inv}
X_if(x)=\frac d{dt}f(x*  t e_i)\Big|_{t=0},\quad i=1,\ldots,m,\quad f\in C^\infty(\R^d).
\end{equation}
Next, we consider the following assumption.

%Hereafter, we will always assume the following assumption holds.
\begin{assumption}\label{homog}\em
We assume the family of vector fields $\{X_1, \dots, X_m\}$ homogeneous of degree $1$ w.r.t. $\D_t$, that is, 
\begin{equation}\label{def-homog}
X_i(f \circ \D_t) = t (X_i f) \circ \D_t, \quad \forall\; t > 0,\; \forall\; f \in C^{\infty}(\R^d)\; \textrm{and}\; i=1, \dots, m.
\end{equation} 
We also assume that $\{X_1, \dots, X_m\}$ satisfies the H\"ormander condition, i.e.,
\begin{equation}\label{Hor}
\text{Lie}(X_1, \dots X_m)(x)=\R^d \quad \text{for all}\,\, x \in \R^d, 
\end{equation}
where $\text{Lie}(X_1, \dots X_m)(x)$ denotes the Lie algebra generated by $\{X_1, \dots, X_m\}$ at $x \in \R^d$.
\end{assumption}
%\begin{assumption}
%Assume that there exists a constant $ M \geq 0$ such that 
%\begin{equation*}
%\div_{\R^d} (X_i) \leq M, \quad \textrm{for any}\; i \in \{1, \dots, m\}. 
%\end{equation*}
%\end{assumption}}

\begin{example}\label{example}
Consider the following vector fields in $\R^4$:
\begin{equation*}
X_1=\partial_{x_1}-\frac {x_2}2\partial_{x_3}-\frac{x_2^2}{12}\partial_{x_4}
\quad\mbox{and}\quad
X_2=\partial_{x_2}+\frac {x_1}2\partial_{x_3}
+\frac {x_1x_2}{12}\partial_{x_4}.
\end{equation*}
Then $X_1$ and $X_2$ provide an example of smooth vector fields with superlinear growth which are left-invariant and satisfy \Cref{homog}. Let us briefly show why.

Let $\GG$ denote the 4-dimensional Lie group  whose Lie algebra $\g$ is endowed with the following bracket rule: if $e_1,\ldots,e_4$ denote the standard orthonormal basis in $\R^4$, then for $1\leq i<j\leq 4$,
$$
[e_1,e_2]=e_3,\quad [e_2,e_3]=e_4 \quad \mbox{and} \quad [e_i,e_j]=0\mbox{ otherwise.}
$$
\def\span{\mbox{\rm Span}}
We define the dilations on $\g$ as follows:
\begin{equation}\label{dilation1}
\D_tx=t(x_1e_2+x_2e_2)+t^2x_3e_3+t^3x_4e_4,\quad t>0,\ x\in \g.
\end{equation}
Then, the stratification $\g = V_1 \oplus V_2 \oplus V_3$ holds with $V_1=\span\{e_1,e_2\}$, $V_2=\span\{e_3\}$. Hence, the associated Lie group $\GG$ is homogeneous and, since $d=4$, $m=2$ and $\kappa=3$, $\GG$ is of Engel type. Moreover, from 
Baker-Campbell-Hausdorff formula,
the group operation is
\begin{equation}\label{star1}
x*  y=
\left(
\begin{array}{c}
x_1+y_1\\
x_2+y_2\\
x_3+y_3+\frac 12 (x_1y_2-x_2y_1)\\
x_4+y_4+\frac 12 (x_2y_3-x_3y_2)
+\frac 1{12}(x_2-y_2)(x_1y_2-x_2y_1)
\end{array}
\right).
\end{equation}
Then we have:
\begin{itemize}
\item[$(i)$]  for $i=1,2$, $X_i$ is the left-invariant vector field associated with $e_i$, as an immediate consequence of the group operation \eqref{star1};
\item[$(ii)$] for $i=1,2$, $X_i$ is homogenous of degree 1 w.r.t. the family of dilations \eqref{dilation1};
\item[$(iii)$] $\{X_1,X_2\}$ fulfills the H\"ormander condition. Indeed, we have 
$$
\begin{array}{l}
[X_1,X_2]
=\partial_{x_3}+\frac{x_3}3\partial_{x_4},\smallskip\cr
[X_1, [X_1,X_2]]
=\frac 13\partial_{x_4},
\end{array}
$$
which yields $\mathrm{Lie}(X_1,X_2)(x)=\R^4$ for every $x\in\R^4$. 
\end{itemize}
\end{example}

Let us discuss some consequences of \eqref{left-inv} and \eqref{def-homog} (the H\"ormander condition \eqref{Hor} will play a crucial role but we  will see that in \Cref{sect:subRiem}).
We write every $X_i$ in its coordinate form:
\begin{equation*}
X_i = \sum_{k=1}^{d} X_{i, k}(x) \partial_{x_k}, \quad i=1,\ldots,m.
\end{equation*}
As in \Cref{example}, it is a general fact  that, under \eqref{left-inv} and \eqref{def-homog}, all 
$X_{i, k}$ are polynomial functions that are homogeneous of degree $\lambda_{k} - 1$ w.r.t. the dilations $\D_{t}$ and the polynomial $X_{i, k}$ depends only on those variables $x_j$ such that $\lambda_{j} \leq \lambda_k -1$, where $\lambda_1<\cdots<\lambda_m$ are the weights appearing in \eqref{DIL} (see \cite[Proposition 1.3.5 and Remark 1.3.7]{BLU}).
Consequently, 
\begin{equation}\label{divXi}
\div_{\R^d}(X_{i})=0,\quad i=1,\ldots,m,
\end{equation}
where $\div_{\R^d}$ denotes the standard divergence operator in $\R^d$: 
\begin{equation}\label{div_euclidea}
\div_{\R^d} u = \partial_1 u_1 + \dots + \partial_d u_d,
\end{equation}
where $u:\R^d\to \R^d$ 
and $u_i$ denotes the $i$-th component of $u$ for $i=1, \dots, d$.
We recall that, in general, the formal adjoint operator of a vector field $X$ is $X^\star=-X-\div_{\R^d}(X)$. Let us stress that, hereafter, all formal adjoint operators must be understood w.r.t. the Lebesgue measure. Thus, following \eqref{divXi}, for $i=1,\ldots,m$,
\begin{equation}\label{adjoint}
X^\star_i=-X_i.
%\textrm{the formal adjoint operator of}\; X_i\; \textrm{is given by}\; -X_i.
\end{equation}

\begin{remarks}\label{waterfall}\em
We observe that, in what follows, \eqref{def-homog} can be relaxed requiring that the vector fields are $\lambda \geq 1$ homogeneous w.r.t. $\D_t$, that is, $	X_i(f \circ \D_t) = t^\lambda (X_i f) \circ \D_t$, $f \in C^{\infty}(\R^d)$ and $i=1, \dots, m$ (see \cite[Remark 3.1.8 and Lemma 3.1.14]{FisherRuzhansky}). In this case, the polynomial $X_{i, k}$ associated with $X_i$, $i=1,\ldots,m$, is $\lambda_k - \lambda$ homogeneous w.r.t. $\D_t$ and  depends only on those variables $x_j$ such that $\lambda_{j} \leq \lambda_k - \lambda$ (see \cite[Proposition 1.3.5]{BLU}). Finally, we still have that the formal adjoint operator of $X_i$ equals $-X_i$.  
\end{remarks}

From now on, we will always consider a family $\{X_1, \dots, X_m\}$ of smooth horizontal left-invariant vector fields  satisfying  \Cref{homog}. 

\subsection{Sub-Riemannian structure}\label{sect:subRiem}

We endow $\R^d$ with the, so-called, Carnot-Carat\'e\-odo\-ry distance constructed as follows. First, given $a, b \in \R$ we say that an absolutely continuous curve $\gamma: [a,b] \to \R^d$ is horizontal if there exist
$\alpha=(\alpha_1, \dots, \alpha_m)$, with 
$\alpha_1, \ldots, \alpha_m \in L^1(a, b)$, such that 
\begin{equation*}
\dot \gamma(t) = \sum_{j=1}^{m} \alpha_j(t) X_j(\gamma(t)), \quad \text{a.e.}\,\, t \in [a,b]
\end{equation*} 
and the length of $\gamma$ is defined as
\begin{equation*}
\ell(\gamma) = \int_{a}^{b} |\alpha(t)|\ dt,
\end{equation*}
where, hereafter, $|\cdot|$ denotes the standard Euclidean norm. From the H\"ormander condition, a well-known result by Chow \cite{Barilari} states that any two points on $\R^d$ can be connected by an horizontal curve. Hence, the definition of Carnot-Carath\'eodory distance is well-posed and given by
\begin{equation}\label{dSR}
d_{\SR}(x, y) = \inf\{\ell(\gamma): \gamma\,\,\text{is an horizontal curve joining}\,\, x\,\,\text{to}\,\, y\}. 
\end{equation}
One can prove a variational interpretation of the above distance as
\begin{equation*}
d_{\SR}(x, y) = \inf\;\{T > 0 :\, \exists\; \gamma:[0,T]\to \R^d, \text{ horizontal and joining}\,\, x\,\, \text{to}\,\, y\,\, \text{with}\,\, |\gamma(t)| \leq 1\}.
\end{equation*}
When $\kappa > 1$, $\kappa$ denoting the degree of the stratification \eqref{direct}, the Carnot-Carath\'eodory distance is not Lipschitz equivalent to the Euclidean one, as proved in \cite[Theorem 10.64, Theorem 10.67]{Barilari}. Indeed, it is well-known that, for any $x \in \R^d$ and for any neighborhood $U_x$ of $x$ there exists $r(x)\in\N$ such that for any $y\in U_x$ one has
\begin{equation*}
\frac{1}{C}|x-y| \leq d_{\SR}(x,y) \leq C|x-y|^{r(x)},
\end{equation*}
where $r(x) \in \N$ is called the nonholonomic degree at $x \in \R^d$ and is given by the minimum of the degrees of the iterated brackets occurring to fulfill the H\"ormander condition at $x$. 

Using the sub-Riemannian distance one can define a norm on $\R^d$ tailored from the Lie group, that is,
\begin{equation*}
\|x\|_{\SR} = d_{\SR}(0,x). 
\end{equation*}

We finally recall that, from the homogeneity of the vector fields $X_1, \dots, X_m$ and the stratification of $\R^d$, $\|\cdot\|_{\SR}$ is equivalent to the homogeneous norm $\|\cdot\|_{\GG}$ in \eqref{hom-norm-dim}
(for details, see  \cite[Section 5.1]{BLU}).

\subsection{Horizontal functional spaces}
%We recall the definition of H\"older space associated with $\{X_1,\ldots, X_m\}$.

% \blue{For every multi-index $J=(j_1, \dots, j_m) \in \N^m$ let $X^J = X_{\ell_1}^{j_1} \cdots X_{\ell_m}^{j_m}$, where $\ell_1, \cdots, \ell_m$ denotes any arbitrary permutation of the indexes $1, \cdots, m$, and let $|J|=j_1 + \cdots + j_m$ be the length of the multi-index $J$.} \red{questa def \`e impropria perch\'e c'\`e anche la dipendenza dalla permutazione. Propongo la seguente:}
 
In the following, for $n\geq 1$, we consider the multi-index $J=(j_1,\ldots,j_n)\in\{1,\ldots,m\}^n$ of length $|J|=n$  and we define $X^J=X_{j_1}\cdots X_{j_n}$. We allow the case $n=0$, that is the void multi-index $J$ of zero length, and we set $X^J=\mbox{\rm{Id}}$.

 We introduce $C^{0}(\R^d)$ as the set of continuous (possibly unbounded) functions on $\R^d$ and we associate the norm
$\| u\|_{C^{0}(\R^d)}= \| u\|_{L^{\infty}(\R^d)}$. Moreover for $\alpha \in (0,1]$, given $u\in C^{0}(\R^d)$ and $U\subset \R^d$ we define the seminorm 
\begin{equation*}
[u]_{C_{\GG}^{0,\alpha}(U)} = \sup_{\substack{ x, y \in U \\ x \not= y}} \frac{|u(x) - u(y)|}{d_{\SR}(x, y)^{\alpha}},
\end{equation*}
where $d_{\SR}$ is the Carnot-Carath\'eodory distance defined in \eqref{dSR}.
We introduce 
\begin{equation*}
C^{0,\alpha}_{\GG}(\R^d)= \left\{u \in C^0(\R^d): [u]_{C_{\GG}^{0,\alpha}(U)} < \infty, {\textrm{ for every compact }} U\subset \R^d  \right\}
\end{equation*}
and the corresponding norm
$$
\|u\|_{C_{\GG}^{0,\alpha}(\R^d)}:=\| u\|_{C^{0}(\R^d)}+ [u]_{C_{\GG}^{0,\alpha}(\R^d)}.
$$
Similarly, for $r \in \N$ and $\alpha \in (0,1]$ we define 
\begin{equation}\label{Holder_space}
C^{r,\alpha}_{\GG}(\R^d)= \left\{u \in C^0(\R^d): X^{J}u\in C^{0,\alpha}_{\GG}(\R^d), \forall\, J\,:\, |J| \leq r  \right\}
\end{equation}
and the corresponding norm
$$
\|u\|_{C_{\GG}^{r,\alpha}(\R^d)}:=\sum_{0 \leq |J| \leq r} \left(\| X^{J}u\|_{{C^0}(\R^d)}+ [X^{J}u]_{C^{0,\alpha}_{\GG}(\R^d)}\right).$$
The spaces $C^{r, \alpha}_{\GG}(\R^d)$ with the norm $\|\cdot\|_{C_{\GG}^{r,\alpha}(\R^d)}$
are Banach spaces for any $r \in \N$ and any $\alpha \in (0,1]$.

Next, we recall the definition of horizontal Sobolev spaces. Let $r \in \N$ and $1 \leq p \leq \infty$. We define the space 
\begin{equation}\label{Sobolev_space}
W^{r, p}_{\GG}(\R^d) = \left\{u \in L^p (\R^d) : X^J u \in L^p(\R^d), \,\, \forall\, J\,:\, |J| \leq r  \right\}
\end{equation}
endowed with the norm 
\begin{equation*}
\| u \|_{W^{r,p}_{\GG}(\R^d)} = \left(\sum_{|J| \leq r} \int_{\R^d} |X^J u(x)|^{p}\ dx\right)^{\frac{1}{p}},
\end{equation*}
for $r\in [1,\infty)$, $1 \leq p < \infty$. For $p=\infty$ we define the norm
	\begin{equation}\label{Sobolev_space1}
	\| u \|_{W^{r,\infty}_{\GG}(\R^d)} = \sum_{|J| \leq r}  \|X^J u(x)\|_{\infty}.
	\end{equation}
As usual, we set $H^1_\GG(\R^d)=W^{1,2}_{\GG}(\R^d)$. 
Given an open subset $\Omega$ of $\R^d$ we define $C^{r, \alpha}_{\GG}(\Omega)$, resp. $W^{r, p}_{\GG}(\Omega)$ with $1\leq p\leq \infty$,  by replacing $\R^d$ with $\Omega$ in \eqref{Holder_space}, resp. in \eqref{Sobolev_space}-\eqref{Sobolev_space1}. And we set $H^1_\GG(\Omega)=W^{1,2}_{\GG}(\Omega)$. Finally, we denote by $C^{r, \alpha}_{\GG, \text{loc}}(\R^d)$ and $W^{r, p}_{\GG, \text{loc}}(\R^d)$ the ``local horizontal'' H\"older and Sobolev space, respectively, as usual: $u\in C^{r, \alpha}_{\GG, \text{loc}}(\R^d)$, resp. $u\in W^{r, p}_{\GG, \text{loc}}(\R^d)$, if and only if the restriction of $u$ on any open and bounded subset $\Omega$ belongs to $C^{r, \alpha}_{\GG}(\Omega)$, resp. to $W^{r, p}_{\GG}(\Omega)$.
%\red{attenzione: abbiamo definito questi spazi per $\Omega$ aperto, qui invece \`e richiesto compatto. Non possiamo dire ``for every open and bounded subset $\Omega$''???}

\subsection{The subelliptic Fokker-Planck equation}

For $u:\R^d\to \R$, we define the horizontal gradient and the horizontal Laplacian of $u$ as
\begin{equation}\label{DeltaG}
\nabla_{\GG} u= (X_1 u, \dots, X_m u)^{T} \in \R^m\quad\mbox{and}\quad
\Delta_{\GG} u = \sum_{i=1}^{m} X^2_i u \in \R. 
\end{equation}
respectively.
%
%\begin{equation*}
%\nabla_{\GG} u= (X_1 u, \dots, X_m u)^{T} \in \R^m
%\end{equation*}
%and respectively
%\begin{equation*}
%\Delta_{\GG} u = \sum_{i=1}^{m} X^2_i u \in \R. 
%\end{equation*}
%
We also consider the horizontal divergence operator, that is, for any vector-valued function $u: \R^d \to \R^m$,

\begin{equation}\label{divG}
\div_{\GG} u = X_1 u_1 + \dots + X_m u_m
\end{equation}
where $u_i$ denotes the $i$-th component of $u$ for $i=1, \dots, m$. 
%By the way, when $m=d$ we use the notation $\div_{\R^d}$ for the standard divergence operator defined in  \eqref{div_euclidea}.

\smallskip
We are now ready to introduce the equation whose analysis we are interested in:
for $s \geq 0 $ and $x \in \R^d$, we consider  the subelliptic Fokker-Planck equation 
\begin{equation}\label{FP}
\begin{cases}
\partial_t \rho(t, y) - \Delta_{\GG} \rho(t, y) - \div_{\GG}(\beta(t, y) \rho(t, y)) = 0, & (t, y) \in (s, \infty) \times \R^d
\\
%\displaystyle{\lim_{t \downarrow s}} 
\rho(t, y)dy\rightharpoonup
\delta_{\{x\}}(dy)\mbox{ as }t \downarrow s
\end{cases}
\end{equation}
where $\beta: [0,\infty) \times \R^d \to \R$ is a smooth vector field (more precise assumptions will be specified later) and, hereafter, the notation $\rightharpoonup$ denotes the weak convergence. 
%and $\GG^*$ denotes the set of dual operator associated with $\GG$, that is, given $X_i \in V_1$ we have 
%\begin{equation*}
%X_i^* = -X_i - \div_{\R^d}(X_i)
%\end{equation*}
%and $\div(X_i)$ acts by moltiplication on real valued functions. 

%\blue{NOTA: ho modificato le seguenti definizioni \ref{def1} e \ref{def2} tenendo conto che la soluzione parte da un tempo $s\geq 0$, con dato iniziale = massa di Dirac in $x$, e si evolve per $t>s$. Controllare! E poi, domanda: nella def \ref{def1}, perch\'e c'\`e $\delta$? Non basta $\rho \in C((s,\infty); C^{2}_{\GG, loc}(\R^{d}))$ con  $\partial_t\rho\in C((s,\infty); C^{0}_{\GG, loc}(\R^{d}))$???}

\begin{definition}\label{def1}
Let $s\geq 0$ and $x\in\R^d$ be fixed. We say that $\rho$ is a {\bf\em classical solution} to \eqref{FP} if $\rho \in C((s,\infty); C^{2}_{\GG, loc}(\R^{d}))$,  $\partial_t\rho\in C((s,\infty); C(\R^{d}))$ and $\rho$ solves equation \eqref{FP} pointwise.
\end{definition}

In order to obtain analytical results for the solution $\rho$ to \eqref{FP} we need to introduce the following notion of weak solution.

We define $L^2_{loc}((s,\infty); H^{1}_{\GG}(\R^d))$ as the set of functions belonging to $L^2((s,T); H^{1}_{\GG}(\R^d))$ for every $T>s$.
% Cos\`i poi evitiamo di scrivere sempre ``per ogni $T>s$''. Che dici? e \textbf{WARNINGS}: 1) se va bene bisogna cambiare dappertutto; 2) già ora bisogna controllare perch\'e mi sembra che ci siano ancora vari $L^2((s,\infty); H^{1}_{\GG}(\R^d))$...}

\begin{definition}\label{def2}
Let $s\geq 0$ and $x\in\R^d$ be fixed. We say that $\rho \in C((s, \infty); H^{1}_{\GG}(\R^d))$ with $\partial_t \rho \in L^2_{loc}((s,\infty); H^{1}_{\GG}(\R^d))$ is an {\bf energy solution} to \eqref{FP} if for any test function  $\varphi \in C((s,\infty); H^{1}_{\GG}(\R^d))$ with $\partial_t \varphi \in L^2_{loc}((s,\infty); H^{1}_{\GG}(\R^d))$ then for every $t \geq s$ it holds:
\begin{multline}\label{energy}
\int_{\R^d}\rho(t,y)\varphi(t,y)\, dy- \varphi(s,x) -
\iint_{(s,t)\times \R^d}\partial_u\varphi (u,y)\rho (u,y) dudy \\ 
+\iint_{(s,t)\times \R^d} \big(\langle \nabla_{\GG}\varphi (u,y) , \nabla_{\GG}\rho (u,y) \rangle+\langle \beta(u,y), \nabla_{\GG} \varphi(u, y)\rangle \rho(u, y) \big)\, dudy =0.
\end{multline}
\end{definition}

\begin{remarks}\em
\begin{enumerate}
	\item 
The weak formulation in \eqref{energy} follows by multiplying equation \eqref{FP} by a test function $\varphi$ satisfying the requests in  \Cref{def2} and then by integrating. In fact, when following this plan, the use of the classical integration by parts formula w.r.t. the space variable would give the adjoint operator of any the vector field $X_i$. But, as already observed in \eqref{adjoint}, under our hypotheses this adjoint operator coincides with $-X_i$, and the integration by parts formula actually returns \eqref{energy}.

\item
%\blue{Furthermore, }
Sometimes we will consider problem \eqref{FP} with initial datum $\rho_s$ different from the Dirac delta. Hence, in this case, we say that $\rho$ is an energy solution of 
\begin{equation*}
\begin{cases}
\partial_t \rho(t, y) - \Delta_{\GG} \rho(t, y) - \div_{\GG}(\beta(t, y) \rho(t, y)) = 0, & (t, y) \in (s, \infty) \times \R^d
\\
\rho(t, y)dy\rightharpoonup
\rho_s(dy)\mbox{ as }t \downarrow s
%\displaystyle{\lim_{t \downarrow s}} \rho(t, y)\;dy=\rho_s(dy)
\end{cases}
\end{equation*}
if for any test function  $\varphi \in C((s,\infty); H^{1}_{\GG}(\R^d))$ with $\partial_t \varphi \in L^2_{loc}((s,\infty); H^{1}_{\GG}(\R^d))$  then for every $t\geq s$ it holds
\begin{multline*}
\int_{\R^d}\rho(t,y)\varphi(t,y)\, dy- \int_{\R^d} \varphi(s,y)\;\rho_s(dy) -
\iint_{(s,t)\times \R^d}\partial_u\varphi (u,y)\rho (u,y) dudy \\ 
+\iint_{(s,t)\times \R^d} \big(\langle \nabla_{\GG}\varphi (u,y) , \nabla_{\GG}\rho (u,y) \rangle+\langle \beta(u,y), \nabla_{\GG} \varphi(u, y)\rangle \rho \big)\, dudy =0.
\end{multline*}
\end{enumerate}

\end{remarks}

%\blue{OSSERVAZIONI.}
%
%\blue{1. nell'ultimo integrale della vecchia def c'\`e un $\sigma$. Che \`e errato, giusto?
%	}
%
%\blue{2. Perch\'e la funzione test $\varphi$ si prende dipendente anche dal tempo? Soluzioni di energia e soluzioni deboli sono oggetti diversi?
%}
%
%\blue{3. Non capisco da dove viene l'ultimo integrale: se uso IBP, mi viene una cosa diversa...
%}
%
%\end{definition}

\section{Analysis of the Fokker-Planck equation}\label{PDE}

We now consider the operator
%Let $\X = \{X_1, \dots, X_m\}$ be a family of smooth vector fields generating a Lie group on $\R^d$ with $d > m$ and let $X_0$ be a smooth drift on $\R^d$. Set
\begin{equation}\label{L}
L = \sum_{i=1}^{m} X_i^2 + X_0=\Delta_{\GG}+X_0,
\end{equation}
where $X_0$ denotes a smooth drift. $L$ turns out to be the infinitesimal generator associated with the diffusion process $\xi$ defined as the solution to the stochastic differential equation (SDE)
\begin{equation}\label{diffusion0}
d\xi_t = X_0(\xi_t)\;dt + \sqrt{2} \sum_{i=1}^{m} X_i(\xi_t) \circ dB^i_t
\end{equation}
where $B=(B^1, \dots, B^d)$ is a standard Brownian motion on $\R^d$ and $\circ$ denotes the Stratonovich integral. As the H\"ormander condition \eqref{Hor} does hold, it is well known (see e.g. \cite[Theorem 2.3.3]{Nua06}) that if $X_0,X_1,\ldots,X_m$ have bounded derivatives of all orders then the transition density $q(s,t,x,y)$ exists and is infinitely differentiable. Moreover, for every fixed $s\geq 0$ and $x\in\R^d$, $(t,y)\mapsto q(s,t,x,y)$ solves the Fokker-Planck equation
\begin{equation}\label{prFP}
\begin{cases}
\partial_t q(s,x,t,y) - L^{\star} q(s,x,t,y) = 0, \quad (t, y) \in (s, \infty) \times \R^d
\\
%\displaystyle{\lim_{t \downarrow s}} \,q(s,x,t,y)\;dy=\delta_{\{x\}}(dy),
q(s,x,t,y)dy\rightharpoonup
\delta_{\{x\}}(dy)\mbox{ as }t \downarrow s
\end{cases}
\end{equation}
where the above limit has to be intended weakly (in law)
and $L^{\star}$ denotes the adjoint operator of $L$ (we will discuss it deeply later on).

Our aim is to set a link between the above ``probabilistic'' Fokker-Planck equation \eqref{prFP} and the ``analytical'' one defined in  \eqref{FP}. As a consequence, we will be able to interpret the solution $\rho$ to \eqref{FP} as the density of a specific probability law, based on the solution $\xi$ to the SDE \eqref{diffusion0} with a drift $X_0$ that might depend also on the time variable, as we will see later on. We stress that the main difference with the existing literature on the argument is the polynomial growth of the vector fields involved in the infinitesimal generator \eqref{L}. Indeed, as for instance in \cite{Nua06} and references therein, it is classically assumed that the vector fields $\{X_1, \dots, X_m\}$ to be at least Lipschitz continuous and, thus, to have sublinear growth.

%\subsection{The PDEs approach}

This section is devoted to the analysis of the Fokker-Planck equation \eqref{FP} from a PDEs viewpoint and this will be the starting point for the probabilistic interpretation of the dynamics described by such an equation.  

Before proceeding with the first main result of this work, we introduce the Fortet-Mourier distance $d_0$ tailored for homogeneous Lie groups: for any two probability  measures $m$ and $m'$ on $\R^d$,
\begin{equation}\label{dFM}
d_{0}(m ,m') = \sup\left\{\int_{\R^d} f(x)\big(dm(x)-dm'(x)\big) : f \in C^{0,1}_{\GG}(\R^d)\, \text{s.t.}\,\,  \|f\|_{C^{0,1}_{\GG}(\R^d)} \leq 1 \right\}. 
\end{equation}

Next we denote by $[f]_{-}$ the negative part of $f$.

%\blue{NOTA: vorrei mantenere l' eq.ne per $t\geq s$, quindi nel teorema che segue ho preso $s$ come istante iniziale e ho cercato di ``sostituire" $T$ con $+\infty$ . Controllare che sia tutto ok. In caso contrario, riscriviamo come nella versione precedente con $T$ dicendo, per\`o, che \` e tutto vero \textit{per ogni} $T>s$.}

\begin{theorem}\label{thm:ex_FP1}
For $s\geq 0$, let $\rho_s(\cdot) \in L^\infty(\R^d)\cap L^1(\R^d) $ be a nonnegative function such that $\int_{\R^d}\rho_s(y) dy=1$ and, for some $\alpha\in(0,1]$, let  $\beta \in C([s,\infty); (C^{1,\alpha}_{\GG}(\R^{d}))^m)$ satisfy
\begin{equation}\label{negative}
\sup_{t\in[s,T]}\|\beta(t,\cdot)\|_{(L^\infty(\R^{d}))^m} <+\infty \;\; \textrm{and} \; \sup_{t\in [s,T]}\|[\div_{\GG}\beta(t,\cdot)]_{-}\|_{L^{\infty}(\R^d)} < \infty
\end{equation}
for any $T > s$. 
%\blue{Assume that there exists a constant $ M \geq 0$ such that 
%\begin{equation*}
%\div_{\R^d} (X_i) \leq M, \quad \textrm{for any}\; i \in \{1, \dots, m\}. 
%\end{equation*}}
Then, the equation
\begin{equation}\label{FP_regular}
\begin{cases}
\partial_t \rho(t, y) - \Delta_{\GG} \rho(t, y) - \div_{\GG}(\beta(t, y) \rho(t, y)) = 0, & (t, y) \in (0, \infty) \times \R^d
\\
\rho(s, y) = \rho_s(y), & y \in \R^d
\end{cases}
\end{equation}
has a unique energy solution  $\rho \in C((s, \infty); H^{1}_{\GG}(\R^d))$ with $\partial_t\rho \in L^2_{loc}((s,\infty); H^{1}_{\GG}(\R^d))$.

Moreover, the following properties hold. 
\begin{itemize}
\item[($i$)] $0\leq \rho(t, y)\leq \|\rho_s(\cdot)\|_{L^\infty(\R^d)}$ for any $(t, y) \in [s, \infty) \times \R^d$ and $\int_{\R^d}\rho(t, y)\, dy=1$ for every $t\geq s$.
\item[($ii$)]  If 
\[
\sup_{t\in[s,T]}\|\beta(t,\cdot)\|_{(C^{1}_{\GG}(\R^{d}))^m} <+\infty
\]
for any $T > s$, then $\rho \in C((s, \infty); C^{2, \alpha}_{\GG, \textrm{loc}}(\R^d))$ is a classical solution of \eqref{FP_regular}.
\item[($iii$)] The map $t\mapsto \rho(t,dy)=\rho(t,y)dy$, $t\geq s$, is locally $\frac 12$-H\"older continuous  w.r.t. the Fortet -Mourier distance: for every $t_1,t_2\in[s,T]$,
\begin{equation*}
	d_0 (\rho(t_1,\cdot),\rho(t_2,\cdot)) \leq 4C_{s,T}(\beta)|t_1-t_2|^{\frac{1}{2}}, 
	\end{equation*}
	where $C_{s,T}(\beta)=1+\displaystyle{\sup_{r \in [s, T]}} \|\beta(r, \cdot)\|_{(L^\infty(\R^d))^m}$.

%\blue{occhio: dalla dim mi sembra che l'affermazione corretta sia la seguente: for every $T>s$ and for every  $t_1,t_2\in[s,T]$ one has
%	\begin{equation*}
%	d_0 (\rho(t_1,\cdot),\rho(t_2,\cdot)) \leq 4C_{s,T}(\beta)|t_1-t_2|^{\frac{1}{2}},
%	\end{equation*}
%	where
%	$C_{s,T}(\beta)=1+\sup_{r \in [s, T]} \|\beta(r, \cdot)\|_{L^\infty(\R^d)}$
%}
\end{itemize}
\end{theorem}
\proof 

We divide the proof in several steps. We, first, show the existence of an energy solution appealing to J.L. Lions Theorem (e.g., \cite[Theorem X.9]{Brezis}) and, then, we show that if 
\begin{equation*}
\sup_{t\in [s,T]}\|[\div_{\GG}\beta(t,\cdot)]_{-}\|_{L^{\infty}(\R^d)} < \infty
\end{equation*}
for any $T > s$ then such an energy solution is also unique. Then, we proceed with the proof of ($i$), ($ii$) and ($iii$). 

\medskip
\fbox{Existence.} Set
\[
H=C([s, \infty); H^1_{\GG}(\R^d)) \cap L^2_{loc}((s,\infty); H^{1}_{\GG}(\R^d))
\]
and for $t \geq s$ define the bilinear form 
\begin{equation*}
B: H \times H  \to \R
\end{equation*}
 by
\begin{multline*}
B(\varphi, \psi) = -\iint_{(s,t)\times \R^d}\partial_u\varphi (u,y)\psi (u,y) dudy \\ 
+\iint_{(s,t)\times \R^d} \big[\langle \nabla_{\GG}\varphi (u,y) , \nabla_{\GG}\psi (u,y) \rangle+\langle \beta(u,y), \nabla_{\GG} \varphi(u, y)\rangle \psi \big]\, dudy.
\end{multline*}
Then, our assumptions \eqref{negative} allow us to develop the same arguments in  \cite[Lemma 3.5]{Federica} in order to deduce the coercivity of the bilinear form $B$. Consequently, J.L. Lions Theorem \cite[Th\'eor\`em X.9]{Brezis} can be applied and we obtain that there exists an energy solution of \eqref{FP_regular}. In particular, rewriting the equation as a subelliptic heat equation of the form 
\begin{equation*}%\label{FP_regular-bis}
\begin{cases}
\partial_t \rho(t, y) - \Delta_{\GG} \rho(t, y) = \div_{\GG}(\beta(t, y) \rho(t, y)), & (t, y) \in (s, \infty) \times \R^d
\\
\rho(s, y) = \rho_s(y), & y \in \R^d
\end{cases}
\end{equation*}
we also obtain, from \cite[Theorem 18]{Acta}, that such a solution satisfies  $X_i X_j \rho \in L^{2} ((s, \infty) \times \R^d)$ and, consequently, $\partial_t \rho \in L^2_{loc}((s,\infty); H^{1}_{\GG}(\R^d))$.

\medskip
\fbox{Uniqueness.}  Let $\rho_1$, $\rho_2$ be two energy solutions to \eqref{FP_regular}, with the same initial datum at time $s$,  and set 
\[
\mu= \rho_1 - \rho_2.
\]
 Then, for any $\varphi\in C([s,\infty) ;H^{1}_{\GG}(\R^d))$ with $\partial_t\varphi\in L^2_{loc}((s,\infty); H^{1}_{\GG}(\R^d))$ 
% \red{ forse: with $\partial_t\varphi\in L^2([s, T]; H^{1}_{\GG}(\R^d))$ for every $T>s$???}
 we have
\begin{multline*}
\int_{\R^d}\mu(x,t)\varphi(x,t)\, dx - \iint_{(s,t)\times \R^d}\partial_t\varphi (t, x)\mu (t, x) dxdt 
\\ +\iint_{(s,t)\times \R^d}\nabla_{\GG}\varphi (t, x)\left(\nabla_{\GG}\mu (t, x)+\beta (t, x)\mu\right)\, dxdt =0
\end{multline*}
where we have used that $\mu(s, x) = 0$. Then, the change of unknown 
\begin{equation*}
\widetilde\mu = e^{-\lambda t} \mu
\end{equation*}
leads to the equation
\begin{equation}\label{exp_sol}
\begin{cases}
\partial_t \widetilde\mu(t, y) - \Delta_{\GG} \widetilde\mu(t, y) - \div_{\GG} (\beta(t, y)\widetilde\mu(t, y)) + \lambda \widetilde\mu(t, y)=0, & (t, y) \in (s, \infty) \times \R^d
\\
\widetilde\mu(s, x) = 0, & x \in \R^d. 
\end{cases} 
\end{equation}
Hence, using $\widetilde\mu$ as a test function for \eqref{exp_sol}, for $t \geq s$ we obtain
\begin{align*}
0 =\ & \int_{s}^{t} \int_{\R^d} \big(\partial_t \widetilde\mu(r, y) - \Delta_{\GG} \widetilde\mu(r, y) - \div_{\GG}(\beta(r, y)\widetilde\mu(r, y)) + \lambda \widetilde\mu(r, y)  \big)\widetilde\mu(r, y) dydr
\\
=\ & \frac{1}{2} \int_{s}^{t} \int_{\R^d} \left(\partial_r \widetilde\mu(r, y)^2 + \left\langle \beta(r, y) ,\nabla_{\GG}  \left(\widetilde\mu(r, y)^2\right) \right\rangle + |\nabla_{\GG} \widetilde\mu(r, y)|^{2} + 2 \lambda \widetilde\mu(r, y)^2\right)\; dydr.
%\\
%\geq\ & \frac{1}{2} \int_{\R^d} \widetilde\mu(t, y)^2\;dy + \left(\lambda - \frac{1}{2}\sup_{z\in [s, \infty)}\|[\div_{\GG}\beta(z,\cdot)]_{-}\|_{L^{\infty}(\R^d)} \right)\int_{s}^{t} \int_{\R^d} \widetilde\mu(r, y)^2\;dydr
\end{align*}
We now fix any $T>s$ and we choose
\begin{equation*}
\lambda > \frac{1}{2}\sup_{t\in[s,T]}\|[\div_{\GG}\beta(t,\cdot)]_{-}\|_{L^{\infty}(\R^d)}.
\end{equation*}
 Then for $t\in[s,T]$ we get
\begin{align*}
0\geq\ & \frac{1}{2} \int_{\R^d} \widetilde\mu(t, y)^2\;dy + \left(\lambda - \frac{1}{2}\sup_{z\in [s,T]}\|[\div_{\GG}\beta(z,\cdot)]_{-}\|_{L^{\infty}(\R^d)} \right)\int_{s}^{t} \int_{\R^d} \widetilde\mu(r, y)^2\;dydr.
\end{align*}
This gives $\widetilde\mu(\cdot, y)\equiv0$ on the time interval $[s,T]$, which yields the uniqueness by construction. 

%
%with
%\begin{equation*}
%\lambda > \frac{1}{2}\sup_{t\in[s,T]}\|[\div_{\GG}\beta(t,\cdot)]_{-}\|_{L^{\infty}(\R^d)},
%\end{equation*}
%leads to the equation
%\begin{equation}\label{exp_sol}
%\begin{cases}
%\partial_t \widetilde\mu(t, y) - \Delta_{\GG} \widetilde\mu(t, y) - \div_{\GG} (\beta(t, y)\widetilde\mu(t, y)) + \lambda \widetilde\mu(t, y)=0, & (t, y) \in (s, \infty) \times \R^d
%\\
%\widetilde\mu(s, x) = 0, & x \in \R^d. 
%\end{cases} 
%\end{equation}
%Hence, using $\widetilde\mu$ as a test function for \eqref{exp_sol}, for $t \geq s$ we obtain
%\begin{align*}
%0 =\ & \int_{s}^{t} \int_{\R^d} \big(\partial_t \widetilde\mu(r, y) - \Delta_{\GG} \widetilde\mu(r, y) - \div_{\GG}(\beta(r, y)\widetilde\mu(r, y)) + \lambda \widetilde\mu(r, y)  \big)\widetilde\mu(r, y) dydr
%\\
%=\ & \frac{1}{2} \int_{s}^{t} \int_{\R^d} \left(\frac{d}{dr} \widetilde\mu(r, y)^2 + \left\langle \beta(r, y) ,\nabla_{\GG}  \left(\widetilde\mu(r, y)^2\right) \right\rangle + |\nabla_{\GG} \widetilde\mu(r, y)|^{2} + 2 \lambda \widetilde\mu(r, y)^2\right)\; dydr
%\\
%\geq\ & \frac{1}{2} \int_{\R^d} \widetilde\mu(t, y)^2\;dy + \left(\lambda - \frac{1}{2}\sup_{z\in [s, \infty)}\|[\div_{\GG}\beta(z,\cdot)]_{-}\|_{L^{\infty}(\R^d)} \right)\int_{s}^{t} \int_{\R^d} \widetilde\mu(r, y)^2\;dydr
%\end{align*}
%which yields the uniqueness by construction. 
%

\medskip
\fbox{Proof of ($i$).} Consider the approximation of problem \eqref{FP_regular} by the same Cauchy problem over the ball $B^{\GG}_R$ centered at $0$ with radius $R$ w.r.t. $d_{\SR}$, with zero Dirichlet condition. First, we observe that by standard comparison principle the solution $\rho_R$ to the restricted problem satisfies
\begin{equation*}
0 \leq \rho_R(t, y) \leq \|\rho_s(\cdot)\|_{L^{\infty}(\R^d)}
\end{equation*}
and, thus, $0 \leq \rho(t, y) \leq \|\rho_s(\cdot)\|_{L^{\infty}(\R^d)}$. 
Next, we integrate the equation on $[s,\infty) \times B^{\GG}_R$, we use the divergence theorem and we note that, by the classical Hopf' Lemma,  $\partial\rho_R/\partial \nu\leq 0$ where $\nu$ is the outward pointing normal to $\partial B^{\GG}_R$. Hence,
we get that $\rho_R(t, \cdot) \in L^1(B^{\GG}_R)$ independently of $R \geq 0$ for any $t \in (s, \infty)$ and, letting $R \uparrow \infty$, we deduce $\rho(t, \cdot) \in L^1(\R^d)$ for any $t \in (s, \infty)$. 

Let us now consider a function $\xi \in C^{\infty}_{c}(\R^d)$ such that $\xi(x) = 1$ for any $x \in B^{\GG}_1$ and $\xi(x) = 0$ for any $x \in \R^d \backslash B^{\GG}_2$, and, for $R>0$, define $\xi_R(x) = \xi\left(\D_{\frac{1}{R}} (x)\right)$, $\D$ denoting the group dilatations.  Hence, being $\xi_R$ smooth we can use $\xi_R$ as a test function for the Fokker-Planck equation in \eqref{FP_regular} as in \Cref{def2}. Then, we get
\begin{multline*}
\int_{\R^d} \rho(t, y)\xi_R(y)\ dy + \iint_{(s, t) \times \R^d} \big(-\Delta_{\GG} \xi_R(y) + \beta(r, y) \nabla_{\GG}\xi_{R}(y) \big)\rho(r, y)\ dr dy  \\ = \int_{\R^d} \rho_s(y)\xi_R (y)\ dy. 
\end{multline*}
Recalling that by \Cref{homog} we have  
\begin{equation*}
\Delta_{\GG} \xi_R(y) = \frac{1}{R^2}\big(\Delta_{\GG} \xi\big)\left(\D_{\frac{1}{R}} (y)\right) \quad \text{and} \quad \nabla_{\GG}\xi_R(y) = \frac{1}{R} \big(\nabla_{\GG}\xi\big)\left(\D_{\frac{1}{R}} (y)\right),
\end{equation*}
by dominated convergence theorem as $R \uparrow \infty$, we conclude
\begin{equation*}
\int_{\R^d} \rho(t, y)\ dy = \int_{\R^d} \rho_s(y)\ dy, \quad \forall\; t \geq s.
\end{equation*}

\medskip
\fbox{Proof of ($ii$).} Consider, again,  the approximation of problem \eqref{FP_regular} by the same Cauchy problem over $B^{\GG}_{\eps}$ with zero Dirichlet condition. We claim that the solution $\rho_R$ to such a truncated problem satisfies $\rho_R\in C([s,\infty)\times B^{\GG}_{\eps})$ and that, for every domain $\Omega\subset (s,\infty)\times B^{\GG}_{\eps}$, there exist a constant $K(\Omega, R)$ (depending on the assumptions, on $\Omega$ and on $R$) and a constant $K'(\Omega)$ (depending on the assumptions and on $\Omega$), such that
\begin{equation}\label{eq:claim1}
\|\rho_R\|_{C^{2,\alpha}_{\GG}(\Omega)}\leq K(\Omega,R)\quad\textrm{and}\quad
\|\rho_R\|_{C^{2,\alpha}_{\GG}(\Omega)}\leq K'(\Omega) 
\end{equation}
for $R$ sufficiently large (recall that $\alpha$ is the H\"older exponent of $\nabla_{\GG}\beta$).
Indeed, consider a sequence $\{\beta_n\}_n$ of drifts such that $\beta_n\in C^\infty$ and $\beta_n$ uniformly converges to $\beta$ in $[s,\infty)\times B^\GG_R$ as $n\to\infty$. Therefore, by the same arguments as before, the truncated problem with $\beta$ replaced by $\beta_n$ has an energy solution $\rho_{R,n}$. However, applying iteratively \cite[Theorem 18]{Acta}, we infer that $\rho_{R,n}\in C^\infty$ and, by standard comparison principle, we get $\|\rho_{R,n}\|_{L^\infty([s,\infty)\times \R^d)}\leq \|\rho_0\|_{L^\infty(\R^d)}$. Moreover, the results in \cite[Theorem 1.1]{BB} (with $k=0$) ensure that $\rho_{R,n}$ fulfills \eqref{eq:claim1} with  constants $K$ and $K'$, both independent of $n$. Letting $n\to \infty$, we accomplish the proof of our claim~\eqref{eq:claim1}.

\medskip
\fbox{Proof of ($iii$).} We 
consider the smooth function
\begin{equation*}
\xi(x)=\left\{\begin{array}{ll}
\exp\left\{\frac{1}{\|x\|_\GG^{2k!}-1}\right\}&\quad\textrm{if }\|x\|_\GG< 1\\0&\quad\textrm{otherwise.}
\end{array}\right.
\end{equation*}

For $\eps > 0$, let
\begin{equation*}
\xi^{\eps}(x) = \frac{C}{\eps^Q}\xi\left(\D_{\frac{1}{\eps}}(x)\right) \qquad (x \in \R^d),
\end{equation*}
$Q$ denoting the homogeneous dimension (see \eqref{hom-norm-dim}),  be a smooth mollifier with support in $B^{\GG}_{\eps}$ and where the constant $C$ is independent of~$\eps$ and such that $\int \xi^{\eps}(x)dx=C\int\xi(x)dx=1$. Note that, recalling the homogeneity of the norm $\| \cdot \|_{\GG}$ and \Cref{homog} we have
\begin{equation}\label{hom}
X_i \xi^{\eps}(x) = \frac{1}{\eps}(X_i \xi) \Big(\D_{\frac 1\eps} (x)\Big), \quad i = 1, \dots, m. 
\end{equation}
Let $\varphi $ be a real valued function with $\|\varphi\|_{C^{0,1}_{\GG}(\R^d)}\leq 1$  
and let $\varphi_{\eps}(x) = \xi^{\eps} * \varphi(x)$ where the symbol ``$*$'' denotes the convolution based on the operation of the group. Note that, by standard calculus and Lagrange theorem (see~\cite[Theorem 20.3.1]{BLU}, there holds:
\begin{equation*}
\|\varphi-\varphi_{\eps}\|_{C^0(\R^d)}\leq \eps.
\end{equation*}
Then, 
\begin{equation}\label{test_hom}
\int_{\R^d} \varphi_{\eps}(y) \big[\rho(t_2, y) - \rho(t_1, y)\big]\ dy = \int_{t_1}^{t_2} \int_{\R^d} \big[\Delta_{\GG} \varphi_{\eps}(y) + \langle \beta(z, y),  \nabla_{\GG} \varphi_{\eps}(y) \rangle\big]\ \rho(z, y)\ dzdy. 
\end{equation}
First, from \eqref{hom} and by convolution properties, we obtain
\begin{multline}\label{hom_2}
\|\Delta_{\GG} \varphi_{\eps}\|_{C^0(\R^d)} = \left\|\sum_{i=1}^{m} X_i \xi^{\eps} \star X_i \varphi\right\|_{C^0(\R^d)}  \\ = \frac{1}{\eps}\left\|\sum_{i=1}^{m} (X_i \xi) \Big(\D_{\frac 1\eps} (x)\Big) \star X_i \varphi\right\|_{C^0(\R^d)}  \leq \frac{1}{\eps} \| \varphi\|_{C^{0,1}_{\GG}(\R^d)}. 
\end{multline}
Hence, for $T\geq \max(t_1,t_2)$ by using \eqref{hom} and \eqref{hom_2} to estimate \eqref{test_hom} we deduce
\begin{equation*}
\int_{\R^d} \varphi_{\eps}(x) \big[\rho(t_2, x) - \rho(t_1, x)\big]\ dx \leq \frac{2}{\eps}\left(1+\sup_{r \in [s, T]} \|\beta(r, \cdot)\|_{L^\infty(\R^d)}\right)\|\varphi\|_{C^{0,1}_{\GG}(\R^d)}|t_2-t_1|
\end{equation*}
which yields 
\begin{align*}
 \int_{\R^d} \varphi(x) \big[\rho(t_2, x) - \rho(t_1, x)\big]\ dx \\
\leq\ & \int_{\R^d} \varphi_{\eps}(x) \big[\rho(t_2, x) - \rho(t_1, x)\big]\ dx + 2 \| \varphi - \varphi_{\eps}\|_{C^{0}(\R^d)}
\\
\leq & \frac{2}{\eps}\left(1+\sup_{r \in [s, T]} \|\beta(r, \cdot)\|_{L^\infty(\R^d)}\right)\|\varphi\|_{C^{0,1}_{\GG}(\R^d)}|t_2-t_1| + 2 \| \varphi - \varphi_{\eps}\|_{C^{0}(\R^d)}
\\
\leq& 2C_{s,T}(\beta)\left(\frac{1}{\eps}|t_2-t_1| + \eps \right)\|\varphi\|_{C^{0,1}_{\GG}(\R^d)},
\end{align*}
where $C_{s,T}(\beta)=1+\displaystyle{\sup_{r \in [s, T]}} \|\beta(r, \cdot)\|_{C^{0}(\R^d)}$.
In conclusion, minimizing over $\eps >0$, we get
\begin{equation*}
d_0 (\rho(t_2, \cdot), \rho(t_1, \cdot)) \leq 4C_{s,T}(\beta)|t_2-t_1|^{\frac{1}{2}} \quad t_1,t_2\in [s,T]. \eqno\square
\end{equation*}

\medskip
\begin{remarks}\em
The following remarks follow from  \Cref{thm:ex_FP1} and its proof. 
\begin{itemize}
\item[--] The proof of \Cref{thm:ex_FP1} is independent of the degree of homogenity of the vector fields w.r.t. $\D_t$ stated in \Cref{homog}. It only affects the H\"older continuity of the solution $\rho$ to \eqref{FP_regular}: one can in fact prove that
%, that is, the following holds:
\begin{equation*}
d_0 (\rho(t_2, \cdot), \rho(t_1, \cdot)) \leq 4C_{s,T}(\beta, \lambda)|t_2-t_1|^{\frac{1}{\lambda+1}} \quad (\; s < t_1 \leq t_2 	\leq T\;)
\end{equation*}
where $\lambda > 0$ is the degree of homogeneity of $\{X_1, \dots, X_m\}$ w.r.t. $\D_t$.

We also stress that, besides the structural difference of our setting w.r.t. \cite{Saloff1} and \cite{Kim}, the H\"older continuity of the solution to the Fokker-Planck equation is evaluated w.r.t. a distance which will be more natural in the sequel for the probabilistic counterpart.
\item[--] Since the seminal work by R.J. DiPerna and P.-L. Lions \cite{DiPernaLions}, the assumption 
\begin{equation}\label{div}
\sup_{t\in[s,T]}\|[\div_{\GG}\beta(t,\cdot)]_{-}\|_{L^{\infty}(\R^d)} < \infty \quad (\;\forall\; T > s\;)
\end{equation}
is the natural one for the analysis of the Fokker-Planck equation in divergence form. Up to the authors knowledge, such condition has not been investigated in the subelliptic context yet. Here, as for the classical Euclidean setting (as, for instance, \cite{Ambrosio, Figalli} and references therein), we show that it leads to the uniqueness of energy solutions. 
\item[--]
We also observe that the Levi's parametrix method used, for instance, in \cite{Fundamental,Polidoro_94,Kalf_92,Blu_02} can be adapted to our setting in order to prove the existence of the classical solution to \eqref{FP}. However, the method used here allows us to work with weaker assumptions involving only the regularity of the function $\beta$, namely \eqref{div}, to obtain the uniqueness of energy solutions and, moreover, to study the regularity in time of such solutions in terms of Wasserstein distance. We also mention \cite{Vlad_15} for a probabilistic interpretation of the Levi's parametrix method. 
\end{itemize}
\end{remarks}

\begin{lemma}\label{L_2}
Under the assumptions of \Cref{thm:ex_FP1}, the unique energy solution of \eqref{FP_regular} satisfies the following property: for every $t\geq s$, 
\begin{equation}\label{L2_est}
\sup_{r\in[s,t]}\int_{\R^d} |\rho(r, x)|^2\;dx \leq M_t\|\rho_s\|^2_{L^2(\R^d)},
% \quad \forall\; r \in [s, t], \; \textrm{and}\; t \geq s 
\end{equation}
with 
\begin{equation*}
M_t= \frac{1}{2} \sup_{r \in [s, t]} \| \beta(r, \cdot)\|_{C^{0}(\R^d)}.
\end{equation*}
\end{lemma}
\proof Inequality \eqref{L2_est} is a consequence of the Gronwall lemma. Indeed, using the solution $\rho$ itself as a test function for \eqref{FP_regular} we have 
\begin{equation*}
\frac{d}{dr} \int_{\R^d} |\rho(r, x)|^{2}\;dx + \int_{\R^d} |\nabla_{\GG} \rho(r, x)|^2\; dx + 2 \int_{\R^d} \langle \beta(r, x), \nabla_{\GG} \rho(r, x) \rangle \rho(r, x)\; dx= 0
\end{equation*}
for any given $t \geq s$ and $r \in [s, t]$. Hence, by Cauchy-Schwarz inequality and Young's inequality we obtain 
\begin{equation*}
\frac{d}{dr} \int_{\R^d} |\rho(r, x)|^2\; dx \leq \frac{1}{2} \sup_{r \in [s, t]} \|\beta(r, \cdot)\|_{C^{0}(\R^d)} \int_{\R^d} |\rho(r, x)|^2\;dx
\end{equation*}
which yields the conclusion. \qed

\begin{theorem}\label{delta}
	
	Let $s\geq 0$ and $x\in\R^d$ be fixed. For some $\alpha\in(0,1]$,  
		let  $\beta \in C([s,\infty); (C^{1,\alpha}_{\GG}(\R^{d}))^m)$ satisfy \eqref{negative}. Then the Fokker-Planck equation \eqref{FP} has a unique energy solution. 
\end{theorem}
\proof 
It suffices to argue as in the proof of \Cref{thm:ex_FP1} considering a family of solutions $\rho^{\eps}$ of \eqref{FP_regular} with the initial datum  $\rho_s^{\eps}$ where $\rho_s^{\eps}$ is a family of smooth, compactly supported functions with $\int_{\R^d} \rho_s^{\eps}(y)\;dy=1$ and $\rho_s^{\eps} \rightharpoonup \delta_{\{x\}}(dy)$ for a fixed $x \in \R^d$. The result then follows by letting $\varepsilon\to 0$. Moreover,  $\rho_s^{\eps}$ can be chosen such that $\|\rho_s^{\eps}\|_{L^2(\R^d)} \leq M$, for some positive constant $M$ independent of $\eps$. Hence, by \Cref{L_2} and Riesz representation theorem we also have that the limit measure is absolutely continuous w.r.t. the Lebesgue measure with a density in $L^2(\R^d)$. 
\qed

\section{Subelliptic stochastic differential equations}\label{Proba}

Let us now consider the Fokker-Planck equation defined on the Lie group $\GG$ of the form
\begin{equation}\label{Fokker-Planck}
\partial_t \rho(t, y)  - \Delta_{\GG} \rho(t, y) - \div_{\GG} (\beta(t, y) \rho(t, y)) = 0, \quad (t, y) \in (s, \infty) \times \R^d
\end{equation}
where $\beta: (0, \infty) \times \R^d \to \R^ m$ is a suitable smooth vector field. Our aim is to give a probabilistic meaning for $\rho$. Roughly speaking, taking in mind the ``probabilistic'' Fokker-Planck equation \eqref{prFP}, we look for $X_0, \dots, X_m$ such that the above second order operator is the adjoint of the infinitesimal generator (possibly time dependent) of a diffusion process $\xi$, that is,
\begin{equation*}
\Delta_{\GG} \rho + \div_{\GG}(\beta\rho) = L^{\star}\rho,
\end{equation*}
$L$ being given in \eqref{L}. 

%But this is not possible, as it can be easily argued. However, we show below that we can find a time dependent drift $X_0(t)$ and a function $\alpha\,:\,[0,\infty)\times \R^d\to \R$ such that 
%\begin{equation*}
%\Delta_{\GG} \rho + \div_{\GG}(\beta\rho) = (L + \alpha)^{\star}\rho,
%\end{equation*}
%where $L$ is defined in \eqref{prFP} with $X_0$ replaced by $X_0(t)$. So, we obtain an (inhomogeneous) infinitesimal generator but, in order to put things in the right way, it has to be ``corrected'' with a multiplicative operator.

\begin{lemma}\label{generator}
Let $L$ be the infinitesimal generator in \eqref{L}:
$$
L = \sum_{i=1}^{m} X_i^2 + X_0
=\Delta_{\GG}+X_0,
$$
where $X_0$ denotes a vector field whose coordinates belong to $C^1_{\GG}(\R^d)$.
Then the formal adjoint operator of $L$ is given by
\begin{equation*}
L^{\star} = \sum_{i=1}^{m} X_i^2 - X_0-\div_{\R^d}(X_0).
\end{equation*}
In particular, if 
\begin{equation*}
X_0 = -\sum_{i=1}^{m}  \beta_i X_i,
%&\alpha = -\div_{\GG}(\beta).
\end{equation*}
with $\beta\in (C^1_{\GG}(\R^d))^m$
% and $X_1,\ldots,X_m$ satisfy the requests in \Cref{Lie}, 
then
\begin{equation*}
L^{\star} \rho =\Delta_{\GG} \rho+ \div_{\GG}(\beta \rho).  
\end{equation*}
And if $\beta$ depends on the time variable $t$ then the same holds for the vector field $X_0$, so we write $X_0(t)$.

\end{lemma}

%\blue{NOTA: se il sospetto a pagina \pageref{pg-dubbio} \`e fondato, $\div_{\R^d}(X_i)=c_i$, quindi
%\begin{equation*}
%\widehat{X}_0 = 2 \sum_{i=1}^{m} c_i - X_0 \quad\mbox{e}\quad
%\widehat\alpha = \alpha + \sum_{i=1}^{m} c_i^2 - \div_{\R^d} X_0.
%\end{equation*}
%}

\proof 
Let $f$ and $g$ be two smooth test functions. Then, we have
\begin{equation*}
\int_{\R^d} Lg \cdot f\;dx = \underbrace{\int_{\R^d} \sum_{i=1}^{m} X_i^2 g \cdot f\;dx}_{\bf A} + \underbrace{\int_{\R^d} X_0 g\cdot f\;dx}_{\bf B}  .
%+\int_{\R^d} \alpha g \cdot f\;dx.
\end{equation*}
Now, by standard computations, $X_0^\star=-X_0-\div_{\R^d}X_0$. Moreover, \eqref{adjoint} gives $X_i^\star=-X_i$, $i=1,\ldots,m$, so that
\begin{align*}
{\bf A}&=  \int_{\R^d} g \sum_{i=1}^{m}X_i^2 f\;dx,\\
{\bf B}&= - \int_{\R^d} g (\div_{\R^d} (X_0) + X_0 )f\;dx,
\end{align*}
%so that
%\begin{equation*}
%\int_{\R^d} Lg \cdot f\;dx  = \int_{\R^d} g \cdot \left(\sum_{j=1}^{m} X_j^2  - X_0 - \div_{\R^d}(X_0) \right)f\;dx.
%\end{equation*}
which yields the first result. 
We take now $X_0=-\sum_{i=1}^m \beta_iX_i$. Then, 
	$$
X_0^\star\rho	=(-X_0-\text{div}_{\R^d}X_0)\rho=\text{div}_{\GG}(\beta \rho)+\rho\sum_{i=1}^m\text{div}_{\R^d}\beta_i\text{div}_{\R^d}X_i
$$
the latter equality following from basic computations. So, by using \eqref{divXi}, the statement follows. \qed

\smallskip

		\smallskip
		
Coming back to the ``analytical'' Fokker-Planck equation \eqref{FP}, we consider the function $\beta\,:\,(0,+\infty)\times \R^d\to \R^m$ appearing therein and we write $\beta(t)=\beta(t,\cdot)$. We now set $X_0(t)$ the vector field in \Cref{generator} with such $\beta$, that is,
\begin{equation}\label{coefficients}
X_0(t) = -\sum_{i=1}^{m} \beta_i(t) X_i.
%=-\text{div}_{\GG}(\beta \cdot).
\end{equation}
We now consider the (inhomogeneous) diffusion process $\xi$ defined as the solution to the SDE
\begin{equation}\label{diffusion}
d\xi_t = X_0(t, \xi_t)\;dt + \sqrt{2} \sum_{i=1}^{m} X_i(\xi_t) \circ dB^i_t,
\end{equation}
whose  infinitesimal generator is
\begin{equation}\label{Lt}
L_t=\sum_{i=1}^m X_i^2+X_0(t)
%=\Delta_{\GG}-\text{div}_{\GG}(\beta \,\cdot).
=\Delta_{\GG}+X_0(t).
\end{equation}
We stress the link with the ``analytical'' Fokker-Plank equation \eqref{Fokker-Planck}:
by \Cref{generator}, the choice \eqref{coefficients} gives
\begin{equation}\label{Lstar}
L_t^{\star} \rho =\Delta_{\GG} \rho+ \div_{\GG}(\beta(t) \rho),
\end{equation}
that is just the operator appearing in \eqref{Fokker-Planck}.

We first study the existence of a unique strong solution to the SDE equation \eqref{diffusion}, which would be immediate if the coordinate functions of $X_0(t), X_1,$ $\ldots,X_m$ had linear growth. Notice that if $\beta$ does not depend on the space variable then the particular structure of the vector fields $X_1,\ldots,X_n$ allows one to construct the unique and non exploding solution by iteration (see e.g. \cite{Roynette}). In the general case, the continuity of $\beta$ just allows to state that  \eqref{diffusion} has a  weak solution (see e.g. \cite[Theorem 2.3]{IW}), possibly exploding. 
%
%%
%qui l'articolo:
%@article {Roynette,
%	AUTHOR = {Roynette, B.},
%	TITLE = {Croissance et mouvements browniens d'un groupe de {L}ie
%		nilpotent et simplement connexe},
%	JOURNAL = {Z. Wahrscheinlichkeitstheorie und Verw. Gebiete},
%	FJOURNAL = {Zeitschrift f\"{u}r Wahrscheinlichkeitstheorie und Verwandte
%		Gebiete},
%	VOLUME = {32},
%	YEAR = {1975},
%	PAGES = {133--138},
%	MRCLASS = {60J65},
%	MRNUMBER = {394909},
%	MRREVIEWER = {Yves\ Guivarc'h},
%	DOI = {10.1007/BF00533094},
%	URL = {https://doi.org/10.1007/BF00533094},
%}
%
However, under our hypotheses, in Proposition \ref{moments} we prove that any solution to \eqref{diffusion} does not blow up and that the existence of a unique strong solution actually holds.

Hereafter, $\E^{s,x} $ denotes the expectation under the law $\P^{s,x}$ of a solution $\xi$ of \eqref{diffusion} such that $\xi_s=x$. 

\begin{proposition}
	\label{moments}
Assume that $\beta \in C([0,+\infty); (C^{0}(\R^{d}))^m)$ and, for every $0\leq s<T$,
$$
\sup_{t \in [s, T]}\|\beta(t,\cdot)\|_{(C^{0}(\R^{d}))^m}<\infty,
$$
Then for every $p\geq 1$ and $t>s$ there exist $C>0$  and $a\geq 1$ such that
\begin{equation}\label{Lp-est}
\E^{s,x}\left(\sup_{s\leq r\leq t}\|\xi_r\|_\GG^p\right)
\leq C\big(1+\|x\|_\GG)^{a},
\end{equation}
where $\xi$ is a solution to equation \eqref{diffusion}. As a consequence, there exists a unique  global strong solution to equation \eqref{diffusion}

\end{proposition}

\proof
As the vector fields satisfy the property in  \Cref{waterfall}, one can rewrite \eqref{diffusion} in terms of an It\^o differential as follows:
for $\ell=1,\ldots,\kappa$, 
\begin{equation}\label{pi-ell-csi}
d\pi_\ell(\xi_t)
=\big(-\beta(t,\xi_t)P_\ell(\pi_{\ell-1}(\xi_t))+Q_\ell(\pi_{\ell-1}(\xi_t))\big)dt +\sum_{i=1}^m R_{\ell, i}(\pi_{\ell-1}(\xi_t))dB^i_t
\end{equation}
where $P_1$, $Q_1$ and $R_{1,i}$, $i=1,\ldots, m$, are just constants and for $\ell\geq 2$, $P_\ell$, $Q_\ell$ and $R_{\ell,i}$, $i=1,\ldots,m$, are suitable polynomials
(we stress that the notation $\circ$ in the differential w.r.t. the Brownian motion has disappeared and the above differential denotes now the It\^o's one).

 By \eqref{hom-norm-dim}, it is clear that the statement follows once we prove that, for every $p\geq 1$ and $\ell=1,\ldots,\kappa$,
\begin{equation}\label{app-pi}
\E^{s,x}\left(\sup_{s\leq r\leq t}|\pi_\ell(\xi_r)|^p\right)
\leq C\Big(1+\sum_{i=1}^
{\ell}|\pi_{\ell}(x)|\Big)^{p\gamma_\ell},
\end{equation}
where $\gamma_\ell=\gamma^{\ell-1}$, $\gamma$ denoting the maximum of the degrees of the polynomials $P_\ell$, $Q_\ell$ and $R_{\ell,i}$ in \eqref{pi-ell-csi}, as $\ell, i$ vary, and $C$ denotes a positive constant, possibly depending on $p,t$. Moreover, since $\E^{s,x}$ is just the integral w.r.t. a probability measure, by the H\"older inequality it suffices to prove \eqref{app-pi} for $p\geq 2$.

To this purpose,  we recall the following well known fact, based on the use of the  Burkholder-Davis-Gundy inequality (see e.g. \cite[Theorem 6.10]{IW}). Consider an It\^o process $Z$ of the form
$$
Z_t=\eta+\int_s^t \Gamma(t)dt+\int_s^t\sum_{i=1}^mA_i(t)dB^i_t
$$
and let $\E$ denote expectation under the referring probability measure. For $p\geq 2$, assume that $\E(|\eta|^p)<\infty$ and, for some $L>0$,
\begin{equation}
\label{app-3}
\E\left(\sup_{s\leq r\leq t}|\Gamma(t)|^p\right)
+\sum_{i=1}^m \E\left(\sup_{s\leq r\leq t}|A_i(t)|^p\right)
\leq L.
\end{equation}
Then it holds
\begin{equation}
\label{app-2}
\E\left(\sup_{s\leq r\leq t}|Z_r|^p\right)
\leq c_p\big(1+\E(|\eta|^p)+L(t-s)\big),
\end{equation}
$c_p$ denoting a constant depending on $p$ only. Now, coming back to \eqref{pi-ell-csi}, for $\ell=1$ we have
$$
d\pi_1(\xi_t)
=(-\beta(t,\xi_t)P_1+Q_1)dt +\sum_{i=1}^m R_{1, i}dB^i_t.
$$
Since $\beta$ is bounded on $[s,t]\times\R^d$, we are in the above situation so we get
$$
\E^{s,x}\left(\sup_{s\leq r\leq t}|\pi_1(\xi_r)|^p\right)
\leq C\big(1+|\pi_1(x)|^p).
$$
Consider now $\pi_2(\xi)$. Looking at \eqref{pi-ell-csi}
with $\ell=2$, we have
$$
\begin{array}{l}
|\Gamma(t)|^p
=
|-\beta(t,\xi_t)P_2(\pi_{1}(\xi_t))+Q_2(\pi_{1}(\xi_t))|^p
\leq C(1+|\pi_{1}(\xi_t)|^{p\gamma})\smallskip\\
|A_i(t)|^p
= |R_{2, i}(\pi_{1}(\xi_t))|^p
\leq C(1+|\pi_{1}(\xi_t)|^{p\gamma}),\quad i=1,\ldots,m,
\end{array}
$$
and using \eqref{app-pi} already proved for $\pi_1(\xi)$, \eqref{app-3} holds. Therefore  \eqref{app-2} gives
$$
\E^{s,x}\left(\sup_{s\leq r\leq t}|\pi_2(\xi_r)|^p\right)
\leq C\big(1+|\pi_2(x)|^p+(1+|\pi_1(x)|)^{p\gamma})
\leq C\big(1+|\pi_1(x)|+|\pi_2(x)|)^{p\gamma}.
$$
By iterating the above arguments, 
\eqref{app-pi} follows for every $\ell\leq \kappa$.

Finally, as  $\beta \in C([0,\infty); (C^{1,\alpha}_{\GG}(\R^{d}))^m)$, by using the $L^p$ estimates \eqref{Lp-est} and standard arguments in stochastic calculus involving the Gronwall's lemma, one easily proves that \eqref{diffusion} admits strong uniqueness. Since weak existence and strong uniqueness imply that there exists a unique strong solution (see e.g. \cite{IW}), the statement follows.
\qed

\begin{remarks}\em
\Cref{moments} holds also when $\{X_1, \dots, X_m\}$ are $\lambda$-homogeneous vector fields w.r.t. the group dilatations $\D_t$. In fact, the key point is the validity of \eqref{pi-ell-csi}, which would be true as well.
Moreover, by using the Gronwall's lemma, one can relax 
the boundedness property for $\beta$ in 	\Cref{moments} by requiring the following local (in time) linearity (in space) condition: for every $s<T$,
$$
\sup_{(x, t) \in\R^d \times [s, T]}\frac{|\beta(t,x)|}{1+|\pi_1(x)|}<\infty.
$$	
	\end{remarks}

%\begin{remarks}\label{Explosion}
%In general, considering the fact that due to the polynomial growth of the vector fields $X_1, \dots, X_m$ the stochastic process $\xi_t$ could explode, one should consider the explosion time $\mathcal{E}$. However, the measure $p(s,t,x,dy)$ in \eqref{def} is well defined as the expectation in the r.h.s. is finite for every bounded function $\varphi$ with compact support. However, as a consequence of \Cref{thm:ex_FP1}, we will show that $\mathcal{E} = + \infty$ a.s.
%
%\end{remarks}

\subsection{Probabilistic interpretation of the Fokker-Planck equation \eqref{FP}}

Let $\xi$ be the solution to the SDE \eqref{diffusion}. For $t\geq s\geq 0$ and $x\in\R^d$, we set
\begin{equation}\label{def}
\mathrm{P}(s, t, x, A) = \P^{s,x}(\xi_t\in A)=\E^{s,x} \left[ 1_{\xi_t\in A}\right],\quad A\in\mathcal{B}(\R^d).
\end{equation}
We recall that, by \eqref{Lstar}, 
 the theory of Kolmogorov forward equations (see e.g. \cite{stroock_2008}) would suggest that the solution of 
% the ``analytical'' Fokker-Planck equation 
 \eqref{FP} is actually the density of the transition probability in \eqref{def}.
 And this is what we are going to prove.
 
\begin{proposition}\label{Ito}
For  $\alpha\in(0,1]$, let  $\beta \in C([0,\infty); (C^{1,\alpha}_{\GG}(\R^{d}))^m)$ satisfy
\begin{equation}\label{negative}
\sup_{t\in[0,T]}\|\beta(t,\cdot)\|_{(L^\infty(\R^{d}))^m} <+\infty \;\; \textrm{and} \; \sup_{t\in [0, T]}\|[\div_{\GG}\beta(t,\cdot)]_{-}\|_{L^{\infty}(\R^d)} < \infty
\end{equation}	
for any $T > 0$. 
Then for every $t>s\geq 0$ and $x\in\R^d$, the measure $\mathrm{P}(s, t, x, \cdot)$ on $(\R^d,\mathcal{B}(\R^d))$ defined  through \eqref{def} is absolutely continuous w.r.t. the Lebesgue measure: 
$$
\mathrm{P}(s, t, x, dy)=p(s,t,x,y)dy.
$$
Moreover, as a function of the forward variables $(t,y)$, the density $p(s, t, x, y)$ is the energy solution of the Fokker-Planck equation \eqref{FP}, that is, $\rho(t, y)=p(s,x,t,y)$ is the energy solution to
		\begin{equation*}
		\begin{cases}
		\partial_t \rho(t, y) - \Delta_{\GG} \rho(t, y) - \div_{\GG}\big(b(t, y)\rho(t, y)\big) = 0, & (t, y) \in (s, \infty) \times \R^d,
		\\
		\displaystyle
		\rho(t, y)dy\rightharpoonup \delta_{\{x\}}(dy)
		\mbox{ as } t\downarrow s.
       \end{cases}
		\end{equation*}
		\end{proposition}

		\proof

For $t>s$, 	let 	
$\varphi \in C^{1}((s,t); C^\infty_c(\R^d))$.
We apply the It\^o's  formula on the time interval $(s,t)$ to the process $Z_r = \varphi(r,\xi_r)$ and we write its stochastic It\^o differential (same notation as the proof of Proposition \ref{moments}). So, we obtain 
\begin{align*}
\varphi(t,\xi_{t})
= \varphi(s,\xi_s)+  \int_s^{t}(\partial_r+L_r) \varphi(r,\xi_r) dr + \int_s^{t} \nabla\varphi(r,\xi_r)^T \X(\xi_r)dB_r,
\end{align*}
where $\X(x)=\big[X_1(x), \dots, X_m(x)\big]$ and $L_t$ is given in \eqref{Lt}.
We notice that, on the time interval $[s, t]$, the random function $\nabla\varphi(r,\xi_r)^T \X(\xi_r)$ is bounded, so its It\^o integral  is a martingale and, by passing to the expectation, we obtain
\begin{equation*}
\E^{s, x} \left[  \varphi(t,\xi_t) \right] = \varphi(x) + \E^{s, x} \left[\int_{s}^{t}  (\partial_r+L_r)  \varphi(r,\xi_r) \;dr\right].
\end{equation*}
By using \eqref{def}, we can rewrite the above equality as
\begin{equation}\label{app1-1}
\int_{\R^d}\varphi(t,y)\mathrm{P}(s, t, x, dy)
= \varphi(x) 
+ \int_{s}^{t} \int_{\R^d} (\partial_r+L_r) \varphi(r,y)\mathrm{P}(s, t, x, dy)\;dr.
\end{equation}
Since, by \Cref{moments}, $\xi$ is the unique solution to \eqref{diffusion}, we get that $\mathrm{P}(s, t, x, dy)$ is the unique probability measure such that \eqref{app1-1} holds for every test function $\varphi\in C^{1}((s,t); C^\infty_c(\R^d))$.

Now, let   $\rho$ denote the unique energy solution to \eqref{FP}.
By using integration-by-parts in \eqref{energy} and by recalling \eqref{Lstar}, 
we get	
\begin{equation}\label{app2-2}
\int_{\R^d}\varphi(t, y)\rho(t,y)\, dy= \varphi(x) 
+\int_s^t\int_{\R^d} (\partial_u+L_u)\varphi (y)\rho(u,y)\, dydu.
\end{equation}
So, as $\rho$ uniquely satisfies \eqref{app2-2} for every $\varphi \in C^{1}((s,t); C^\infty_c(\R^d))$ and, by \Cref{thm:ex_FP1},  $\rho(t, y)\;dy$ is a probability measure, by comparing \eqref{app1-1} and \eqref{app2-2} one gets  $\mathrm{P}(s,t,x,dy)=\rho(t,y)dy$.

 \qed

%\blue{tra l'altro, mi pare che in questo modo dimostriamo che se esiste una soluzione debole della \eqref{FP} allora questa \`e unica, perch\'e $\rho(t,y)dy=p(s,t,x,dy)$.}

\begin{remarks}\label{rem-diff}\em 
If $\beta$ does not depend on the space variable, by using the iteration argument in \cite{Roynette} one can get the existence of the transition density of the diffusion in \eqref{diffusion}	by means of a probabilistic approach. If this is not the case, the existence of the probability density of solutions to SDEs is a problem which is classically investigated by means of Malliavin calculus, we refer to \cite{Nua06} and references quoted therein. Typically, one asks for the smoothness of the diffusion coefficients and that the H\"ormander condition holds for $X_1,\ldots,X_m$. But, more significantly, it is required that $X_0, X_1,\ldots,X_m$ have all continuous and \textit{ bounded} derivatives with respect to the space variable, and this is not our case, as $X_1,\ldots,X_m$ are polynomials. This explains why the existence of the transition density given in \Cref{Ito} is not a trivial consequence of  classical results in stochastic calculus.

\end{remarks}		
		
By resuming, we can	state the following result.
			
		\begin{theorem}\label{representation}
		Assume that $\beta \in C([0,\infty); (C^{1,\alpha}_{\GG}(\R^{d}))^m)$, for some $\alpha \in (0,1]$, satisfies 
\begin{equation}\label{negative}
\sup_{t\in[0,T]}\|\beta(t,\cdot)\|_{(L^\infty(\R^{d}))^m} <+\infty \;\; \textrm{and} \; \sup_{t\in [0, T]}\|[\div_{\GG}\beta(t,\cdot)]_{-}\|_{L^{\infty}(\R^d)} < \infty
\end{equation}	
for any $T > 0$. 
		Then, for every $x\in \R^d$ and $s\geq 0$, the unique weak solution $\rho$ to the subelliptic Fokker-Planck equation \eqref{FP} is given by the density in $\R^d$ of the transition probability measure of the process $\xi$ solution to \eqref{diffusion}:
		\begin{equation}\label{process}
		\rho(t,y)=p(s,t,x,y), \text{ with }	p(s,t,x,y)dy=\E^{s,x} \left[{\bf 1}_{\{\xi_{t}\in\; dy\}}\right]
		\end{equation}
satisfies  the problem
$$
\begin{cases}
\partial_t \rho(t,y) - \Delta_{\GG} \rho(t, y) - \div_{\GG}\big(\beta(t, y)\rho(t, y)\big) = 0, \quad (t, y) \in (s, \infty) \times \R^d
\\
		\displaystyle
\rho(t, y)dy\rightharpoonup \delta_{\{x\}}(dy)
\mbox{ as } t\downarrow s
\end{cases}
$$
in the weak sense. Moreover, the following properties hold. 
\begin{itemize}
\item[($i$)] {\bf [H\"older continuity]} The map $t\mapsto p(s,t,x,dy)=p(s,t,x,y)dy$ is locally H\"older continuous with exponent $1/2$ w.r.t. the Fortet-Mourier distance: for every $T\geq s\geq 0$ there exists a constant $C_{s,T}(\beta)$ such that for every $t_1,t_2\in[s,T]$ one has
\begin{equation*}
d_0 (p(s, t_1, x, \cdot),p(s, t_2, x, \cdot)) \leq 4C_{s,T}(\beta) |t_1-t_2|^{\frac{1}{2}},
\end{equation*}
where $C_{s,T}(\beta)=1+\displaystyle{\sup_{r \in [s, T]}} \|\beta(r, \cdot)\|_{(L^\infty(\R^d))^m}$. 
\item[($ii$)] {\bf[Existence of all moments]}
For every $T\geq s\geq 0$ and $q\geq 1$ there exist two constants $C_{q,T}>0$ (depending on $q$ and $T$) and $a_q\geq 1$ (depending on $q$ only) such that 
\begin{equation*}
\sup_{t \in [s, T]}\int\|y\|_\GG^q\,p(s,t,x,y)dy \leq C_{q,T}(1+\|x\|_\GG)^{a_q}.
\end{equation*}
\end{itemize}
 \end{theorem}
 \proof 
 
\eqref{process} is proved in \Cref{Ito}. Moreover, ($i$) follows from ($iii$) in \Cref{thm:ex_FP1} and $(ii)$ is an immediate consequence of \Cref{Ito} and the fact that
\begin{equation*}
\int\|y\|_\GG^q\,p(s,t,x,y)dy
=\E^{s,x}(\|\xi_t\|_{\GG}^q)
\leq \E^{s,x}\left(\sup_{r \in [s, T]}\|\xi_r\|_{\GG}^q\right). \eqno\square
\end{equation*}

\begin{remarks}\em
We recall that, since \Cref{thm:ex_FP1} is independent of the degree of homogenity of the vector fields $\{X_1, \dots, X_m\}$,  we also deduce that the H\"older continuity of $t\mapsto p(s, t, x, \cdot )$ holds in the following sense: if $\{X_1, \dots, X_m\}$ are $\lambda$-homogeneous vector fields w.r.t $\D_t$ we have
 \begin{equation*}
d_0 (p(s, t_1, x, \cdot),p(s, t_2, x, \cdot))  \leq  4C_{s,T}(\beta, \lambda)  |t_1-t_2|^{\frac{1}{\lambda + 1}} \qquad \forall\; t_1, t_2 \in (s,T].
\end{equation*} 
Similarly, as the moment estimates in \Cref{Ito} hold for $\lambda$-homogeneous w.r.t. $\D_t$ vector filed, statement $(ii)$ in \Cref{representation} holds as well.
 \end{remarks}
 
 		\subsection{Feynman-Kac formula}
 
 Having at our disposal the existence of all moments of the subelliptic process $\xi_t$, we next provide the Feynman-Kac formula, part of which has been already found (for instance) in \cite{Fundamental, Schauder, Uniqueness} with analytical techniques.

 \begin{theorem}\label{Feynman-Kac}
 	Let $t>0$  and $L_t$ be given in \eqref{Lt}. Suppose there exists a classical solution $u\in C_{\GG}^{1,2}([0,t)\times \R^d)$ of the parabolic equation
 	\begin{equation}\label{parabolic}
 	\begin{cases}
 	\partial_s u(s, x) + L_s u(s, x) + h(s, x) u(s, x ) = f(s, x), & (s, x) \in [0,t) \times \R^d
 	\\
 	u(t, x) = \psi(x), & x \in \R^d
 	\end{cases}
 	\end{equation}
 	where $h$ is bounded from above and $\psi$ and $f$ are continuous and satisfying the following growth condition: 
 	there exist $L>0$ and $p\geq 0$ such that
 	$$
 	|\psi(x)|+\sup_{s\in [0,t]}|f(s,x)|\leq N(1+\|x\|_{\GG})^p.
 	$$
Assume that the solution $u$ satisfies the same polynomial growth condition. Then, for $(s, x) \in [0,t] \times \R^d$, 	\begin{equation}\label{representation1}
 	u(s, x) = \E^{s, x} \left[\psi(\xi_t)e^{\int_{s}^{t}h(r, \xi_r)\;dr} \right]
 	- \E^{s, x} \left[\int_{s}^{t} f(v, \xi_v)e^{\int_{s}^{v}h(r, \xi_r)\;dr}\;dv \right].
 	\end{equation}
As a consequence, \eqref{representation1} gives the unique solution in the class of the functions having polynomial growth in $x$, uniformly in $s\in[0,t]$. 
 \end{theorem}

 \begin{proof}
 	
Let $s\in[0,t]$ and $x\in\R^d$ be fixed. Let $R>0$ be such that $x\in B^\GG_R$ and let $\tau_R$ denote the exit time of $\xi_t$ from the ball $B^\GG_R$. We apply the It\^o's  formula on the time interval $[s,t\wedge \tau_R]$ to the process $Z_r =e^{\int_s^{r}h(v,\xi_v)dv} u(r,\xi_r)$ and we write the stochastic differential in terms of an It\^o integral (see notations in the proof of \Cref{Ito}). So, we get
 \begin{align*}
e^{\int_s^{t\wedge \tau_R}h(v,\xi_v)dv} u(t\wedge \tau_R,\xi_{t\wedge \tau_R})
=& u(s,\xi_s)
+  \int_s^{t\wedge \tau_R}e^{\int_s^vh(v,\xi_v)dv}(\partial_r +L_r+h)u(r, \xi_r) dr \\
&+ \int_s^{t\wedge \tau_R} e^{\int_s^vh(v,\xi_v)dv}\nabla_xu(r,\xi_r)^T \X(\xi_r)dB_r
%\\
%=& u(s,\xi_s)
%+  \int_s^{t\wedge \tau_R}e^{\int_s^vh(v,\xi_v)dv}f(r, \xi_r) dr \\
%&+ \int_s^{t\wedge \tau_R} e^{\int_s^vh(v,\xi_v)dv}\nabla_xu(r,\xi_r)^T \X(\xi_r)dB_r.
\end{align*}
and by using \eqref{parabolic} we obtain
 \begin{align*}
u(s,\xi_s)
=& 
e^{\int_s^{t\wedge \tau_R}h(v,\xi_v)dv} u(t\wedge \tau_R,\xi_{t\wedge \tau_R})
-  \int_s^{t\wedge \tau_R}e^{\int_s^vh(v,\xi_v)dv}f(r, \xi_r) dr \\
&- \int_s^{t\wedge \tau_R} e^{\int_s^vh(v,\xi_v)dv}\nabla_xu(r,\xi_r)^T \X(\xi_r)dB_r
\end{align*}
We now pass to expectation. Since the integral w.r.t. the Brownian motion is a martingale (the process inside being bounded on $[s,\tau_R]$),  we get
 \begin{align*}
u(s,x)
=&\E^{s,x}\left[e^{\int_s^{t}h(v,\xi_v)dv} \psi(\xi_{t})1_{\{\tau_R>t\}}\right]
+\E^{s,x}\left[e^{\int_s^{\tau_R}h(v,\xi_v)dv} u(\tau_R,\xi_{\tau_R})1_{\{\tau_R\leq t\}}\right]\\
&-\E^{s,x}\left[  \int_s^{t\wedge \tau_R}e^{\int_s^vh(v,\xi_v)dv}f(r, \xi_r) dr\right]
=\gamma^1_R+\gamma^2_R-\gamma^3_R.
\end{align*}
It now suffices to study the asymptotic behavior, as $R\to\infty$, of $\gamma^i_R$ for $i=1,2,3$.

By \Cref{moments},  $\tau_R\uparrow +\infty$. Then, since $h$ is bounded from above and $\psi$ grows polynomially, by \Cref{moments} and the Lebesgue dominated theorem one easily gets 
$$
\gamma^1_R\to \E^{s,x}\left[e^{\int_s^{t}h(v,\xi_v)dv} \psi(\xi_{t})\right].
$$
And with similar arguments one obtains 
$$
\gamma^3_R\to \E^{s,x}\left[  \int_s^{t}e^{\int_s^vh(v,\xi_v)dv}f(r, \xi_r) dr\right].
$$
It remains to prove that $\gamma^2_R\to 0$. In the following, $C$ denotes a constant which can vary from a line to another, can depend on the involved  parameters but is independent of $R$.

First, since $h$ is bounded from above and $u$ grows polynomially in the space variables uniformly in the time variable, for some constants $C>0$ and $p\geq 0$ we get
$$
|\gamma^2_R|\leq C\,\E^{s,x}\Big[(1+\|\xi_{\tau_R}\|_\GG^p)1_{\{\tau_R\leq t\}}\Big]= C(1+R)^p\,\P^{x,s}(\tau_R\leq t).
$$
By the Markov inequality, for any $q\geq 1$  we have
$$
\P^{s,x}(\tau_R<t)
=\P^{s,x}\left(\displaystyle{\sup_{r \in [s, t]}}\|\xi_r\|_\GG>R\right)
\leq \frac{1}{R^q}\E^{s,x}\left(\displaystyle{\sup_{r \in [s, t]}}\|\xi_r\|_\GG^q\right)
\leq \frac{C}{R^q}(1+\|x\|_\GG)^a,
$$
where $C>0$ and $a\geq 1$ depend on $q$ but are independent of $R$. Taking $q=p+1$ and inserting above, we get
$$
 |\gamma^2_R|\leq C\, \frac{(1+R)^p}{R^{p+1}}(1+\|x\|_\GG)^a\to 0\mbox{ as $R\to\infty$}.
$$
The statement now holds.
\end{proof}

%\bibliographystyle{plain}
%\bibliographystyle{unsrt}
%\bibliography{References}

\begin{thebibliography}{10}

\bibitem{Barilari}
Andrei Agrachev, Davide Barilari, and Ugo Boscain.
\newblock {\em A comprehensive introduction to sub-{Riemannian} geometry.
  {From} the {Hamiltonian} viewpoint. {With} an appendix by {Igor} {Zelenko}},
  volume 181 of {\em Camb. Stud. Adv. Math.}
\newblock Cambridge: Cambridge University Press, 2020.

\bibitem{Sub1}
Georgios~K. Alexopoulos.
\newblock {\em Sub-{Laplacians} with drift on {Lie} groups of polynomial volume
  growth}, volume 739 of {\em Mem. Am. Math. Soc.}
\newblock Providence, RI: American Mathematical Society (AMS), 2002.

\bibitem{Ambrosio}
Luigi Ambrosio.
\newblock Transport equation and {Cauchy} problem for {BV} vector fields.
\newblock {\em Invent. Math.}, 158(2):227--260, 2004.

\bibitem{Ancona_96}
Fabio Ancona.
\newblock Decomposition of homogeneous vector fields of degree one and
  representation of the flow.
\newblock {\em Ann. Inst. Henri Poincar{\'e}, Anal. Non Lin{\'e}aire},
  13(2):135--169, 1996.

\bibitem{Vlad_15}
Vlad Bally and Arturo Kohatsu-Higa.
\newblock A probabilistic interpretation of the parametrix method.
\newblock {\em Ann. Appl. Probab.}, 25(6):3095--3138, 2015.

\bibitem{Blu_02}
Andrea Bonfiglioli, Ermanno Lanconelli, and Francesco Uguzzoni.
\newblock Uniform {Gaussian} estimates of the fundamental solutions for heat
  operators on {Carnot} groups.
\newblock {\em Adv. Differ. Equ.}, 7(10):1153--1192, 2002.

\bibitem{Fundamental}
Andrea Bonfiglioli, Ermanno Lanconelli, and Francesco Uguzzoni.
\newblock Fundamental solutions for non-divergence form operators on stratified
  groups.
\newblock {\em Trans. Am. Math. Soc.}, 356(7):2709--2737, 2004.

\bibitem{BLU}
Andrea Bonfiglioli, Ermanno Lanconelli, and Francesco Uguzzoni.
\newblock {\em Stratified {Lie} groups and potential theory for their
  sub-{Laplacians}}.
\newblock Springer Monogr. Math. New York, NY: Springer, 2007.

\bibitem{Bramanti_2010}
Marco Bramanti, Luca Brandolini, Ernesto Lanconelli and Francesco Uguzzoni.
\newblock Non-divergence equations structured on H{\"o}rmander vector fields:
  heat kernels and Harnack inequalities.
\newblock {\em Mem. Amer. Math. Soc.}, 204, (2010) n. 961.

\bibitem{BB}
Marco Bramanti and Luca Brandolini.
\newblock Schauder estimates for parabolic nondivergence operators of
  {H{\"o}rmander} type.
\newblock {\em J. Differ. Equations}, 234(1):177--245, 2007.

\bibitem{Schauder}
Marco Bramanti and Luca Brandolini.
\newblock Schauder estimates for parabolic nondivergence operators of
  {H{\"o}rmander} type.
\newblock {\em J. Differ. Equations}, 234(1):177--245, 2007.

\bibitem{Brezis}
Ha\"im Brezis.
\newblock {\em Analyse fonctionelle}, volume Collection Math\'ematique pour la
  ma\^etrise.
\newblock Masson, 1987.

\bibitem{Uniqueness}
Chiara Cinti.
\newblock Uniqueness in the {Cauchy} problem for a class of hypoelliptic
  ultraparabolic operators.
\newblock {\em Atti Accad. Naz. Lincei, Cl. Sci. Fis. Mat. Nat., IX. Ser.,
  Rend. Lincei, Mat. Appl.}, 20(2):145--158, 2009.

\bibitem{Sub2}
Giovanna Citti, Nicola Garofalo, and Ermanno Lanconelli.
\newblock Harnack's inequality for sum of squares of vector fields plus a
  potential.
\newblock {\em Am. J. Math.}, 115(3):699--734, 1993.

\bibitem{Neuro}
Giovanna Citti and Alessandro Sarti, editors.
\newblock {\em Neuromathematics of vision}.
\newblock Lect. Notes Morphog. Berlin: Springer, 2014.

\bibitem{Multi1}
Emiliano Cristiani, Benedetto Piccoli and Andrea Tosin, 
Multiscale modeling of pedestrian dynamics. 
MSA. Modeling, Simulation and Applications 12. Cham: Springer, xvi, 260 p. (2014).

\bibitem{Teichman}
Christa Cuchiero, Martin Larsson, and Josef Teichmann.
\newblock Deep neural networks, generic universal interpolation, and controlled
  {ODE}s.
\newblock {\em SIAM J. Math. Data Sci.}, 2(3):901--919, 2020.

\bibitem{DiPernaLions}
Ronald~J. DiPerna and Pierre-Louis Lions.
\newblock Ordinary differential equations, transport theory and {Sobolev}
  spaces.
\newblock {\em Invent. Math.}, 98(3):511--547, 1989.


\bibitem{Federica}
Federica Dragoni and Ermal Feleqi,
Ergodic mean field games with H\"ormander diffusions. 
Calc. Var. Partial Differ. Equ. 57, No. 5, Paper No. 116, 22 p. (2018).

\bibitem{Ermal}
Ermal Feleqi, Diogo Gomes and Tada, Teruo
Hypoelliptic mean field games: a case study. 
Minimax Theory Appl. 5, No. 2, 305-326 (2020).

\bibitem{Figalli}
Alessio Figalli.
\newblock Existence and uniqueness of martingale solutions for {SDEs} with
  rough or degenerate coefficients.
\newblock {\em J. Funct. Anal.}, 254(1):109--153, 2008.

\bibitem{FisherRuzhansky}
Veronique Fischer and Michael Ruzhansky.
\newblock {\em Quantization on nilpotent {L}ie groups}, volume 314 of {\em
  Progress in Mathematics}.
\newblock Birkh\"{a}user/Springer, [Cham], 2016.

\bibitem{NEW}
Chin-Wei Huang, Milad Aghajohari, Avishek~Joey Bose, Prakash Panangaden, and
  Aaron Courville.
\newblock Riemannian diffusion models.
\newblock {\em NeurIPS 2022}, 2022.

\bibitem{NEW1}
Chin-Wei Huang, Jae Hyun Lim, Aaron Courville.
\newblock A Variational Perspective on Diffusion-Based Generative Models and Score Matching.
\newblock 35th Conference on Neural Information Processing Systems (NeurIPS 2021), Sydney, Australia.

\bibitem{IW}
Nobuyuki Ikeda and Shinzo Watanabe.
\newblock {\em Stochastic Differential Equations and Diffusion Processes},
  volume~24 of {\em North-Holland Mathematical Library}.
\newblock Elsevier, 1981.

\bibitem{Jean}
Fr{\'e}d{\'e}ric Jean.
\newblock {\em Control of nonholonomic systems: from sub-{Riemannian} geometry
  to motion planning}.
\newblock SpringerBriefs Math. Cham: Springer; Bilbao: BCAM -- Basque Center
  for Applied Mathematics, 2014.

\bibitem{Sub3}
David~S. Jerison and Antonio S{\'a}nchez-Calle.
\newblock Estimates for the heat kernel for a sum of squares of vector fields.
\newblock {\em Indiana Univ. Math. J.}, 35:835--854, 1986.

\bibitem{Sub4}
Hubert Kalf.
\newblock On {E}. {E}. {Levi}'s method of constructing a fundamental solution
  for second- order elliptic equations.
\newblock {\em Rend. Circ. Mat. Palermo (2)}, 41(2):251--294, 1992.

\bibitem{Kalf_92}
Hubert Kalf.
\newblock On {E}. {E}. {Levi}'s method of constructing a fundamental solution
  for second- order elliptic equations.
\newblock {\em Rend. Circ. Mat. Palermo (2)}, 41(2):251--294, 1992.



\bibitem{Mannucci2}
Paola Mannucci, Claudio Marchi, Carlo Mariconda, and Nicoletta Tchou.
\newblock Non-coercive first order mean field games.
\newblock {\em J. Differ. Equations}, 269(5):4503--4543, 2020.

\bibitem{Mendico_MathAnn}
Paola Mannucci, Claudio Marchi, and Cristian Mendico.
\newblock Semi-linear parabolic equations on homogeneous lie groups arising
  from mean field games.
\newblock {\em ArXiv:2307.07257 (To appear in Mathematische Annalen)}, 2024.

\bibitem{Mannucci1}
Paola Mannucci, Claudio Marchi, and Nicoletta Tchou.
\newblock Non coercive unbounded first order mean field games: the {Heisenberg}
  example.
\newblock {\em J. Differ. Equations}, 309:809--840, 2022.


\bibitem{Stat}
Kanti V. Mardia, K. V. and  Peter E. Jupp, 
Directional statistics, 
volume 494. John Wiley \& Sons, 2009.

\bibitem{Montgomery}
Richard Montgomery.
\newblock {\em A tour of subriemannian geometries, their geodesics and
  applications}, volume~91 of {\em Math. Surv. Monogr.}
\newblock Providence, RI: American Mathematical Society (AMS), 2002.

\bibitem{Nua06}
David Nualart.
\newblock {\em The {Malliavin} calculus and related topics.}
\newblock Probab. Appl. Berlin: Springer, 2nd ed. edition, 2006.

\bibitem{Polidoro_94}
Sergio Polidoro.
\newblock On a class of ultraparabolic operators of
  {Kolmogorov}-{Fokker}-{Planck} type.
\newblock {\em Matematiche}, 49(1):53--105, 1994.

\bibitem{Acta}
Linda~Preiss Rothschild and Elias~M. Stein.
\newblock Hypoelliptic differential operators and nilpotent groups.
\newblock {\em Acta Math.}, 137:247--320, 1977.

\bibitem{Roynette}
Bernard Roynette. 
\newblock
Croissance et mouvements browniens d'un groupe de {L}ie
nilpotent et simplement connexe
\newblock 
{\em
Z. Wahrscheinlichkeitstheorie und Verw. Gebiete},
32:133--138, 1975.

\bibitem{Filippo}
 Filippo Santambrogio,
Optimal transport for applied mathematicians. Calculus of variations, PDEs, and modeling. 
Progress in Nonlinear Differential Equations and Their Applications 87. Cham: Birkh\"auser/Springer. xxvii, 353 p. (2015).

\bibitem{stroock_2008}
Daniel~W. Stroock.
\newblock {\em Partial Differential Equations for Probabilists}.
\newblock Cambridge Studies in Advanced Mathematics. Cambridge University
  Press, 2008.

\bibitem{Sub5}
Francesco Uguzzoni.
\newblock A note on a generalized form of the {Laplacian} and of
  sub-{Laplacians}.
\newblock {\em Arch. Math.}, 80(5):516--524, 2003.

\bibitem{Mendico_JDE}
Piermarco Cannarsa and Cristian Mendico.
\newblock Asymptotic analysis for {Hamilton}-{Jacobi}-{Bellman} equations on
  {Euclidean} space.
\newblock {\em J. Differ. Equations}, 332:83--122, 2022.


%\bibitem{Lu_1992}
%Guozhen Lu
%\newblock Weighted Poincar{\'{e}} and Sobolev inequalities for vector fields
%  satisfying H{\"o}rmander's condition and applications.
%\newblock {\em Rev. Mat. Iberoamericana}, (1992) 367--439.

\bibitem{Lanconelli}
Ermanno Lanconelli, Andrea Pascucci,  Sergio Polidoro, 
Linear and nonlinear ultraparabolic equations of Kolmogorov type arising in diffusion theory and in finance. 
Birman, Michael Sh. (ed.) et al., Nonlinear problems in mathematical physics and related topics II. In honour of Professor O. A. Ladyzhenskaya. New York, NY: Kluwer Academic Publishers. Int. Math. Ser., N.Y. 2, 243-265 (2002).

\bibitem{Anceschi}
Francesca Anceschi, Sergio Polidoro,
A survey on the classical theory for Kolmogorov equation. 
Matematiche 75, No. 1, 221-258 (2020).

\bibitem{Pascucci}
Marco Di Francesco, Andrea Pascucci,
On a class of degenerate parabolic equations of Kolmogorov type.
AMRX, Appl. Math. Res. Express 2005, No. 3, 116 p. (2005).

\bibitem{Saloff}
 Laurent Saloff-Coste,
Parabolic Harnack inequality for divergence form second order differential operators. 
Potential Anal. 4, No. 4, 429-467 (1995).

\bibitem{Saloff1}
Janna Lierl,
Parabolic Harnack inequality for time-dependent non-symmetric Dirichlet forms. 
J. Math. Pures Appl. (9) 140, 1-66 (2020).


\bibitem{Kim}
Daehong Kim and  Yoichi Oshima, 
On the upper rate functions of some time inhomogeneous diffusion processes. 
Potential Anal. 60, No. 3, 1181-1213 (2024).


\end{thebibliography}

%\bibitem[Nua06]{Nua} D. Nualart. \textit{The Malliavin calculus and related topics. Second edition}. Springer-Verlag, Berlin, 2006.

\end{document}